\documentclass[11pt]{article}
\usepackage{amsfonts, amstext, latexsym}
\usepackage{fullpage}
\usepackage[dvips]{graphicx}
\usepackage[dvips]{color}
\usepackage{amsmath, amsthm}
\usepackage{psfrag}

\usepackage{color}

\newtheorem{proposition}{Proposition}
\newtheorem{definition}{Definition}
\newtheorem{theorem}{Theorem}
\newtheorem{corollary}{Corollary}
\newtheorem{lemma}{Lemma}

\newtheorem{example}{Example}

\graphicspath {
    {./}
    {./eps/}
}

\DeclareMathOperator{\conv}{conv}

\newcommand{\eps}{\varepsilon}
\newcommand{\R}{\mathbb{R}}
\begin{document}
\title {Convergence Analysis of Meshfree Approximation Schemes}
\author{A.~Bompadre${}^{1}$ \and B.~Schmidt${}^{2}$ \and M.~Ortiz${}^{3}$}

\addtocounter{footnote}{1}\footnotetext{Graduate Aerospace
Laboratories, California Institute of Technology, Pasadena, CA
91125, United States, {\tt abompadr@caltech.edu}}

\addtocounter{footnote}{1}\footnotetext{Zentrum Mathematik,
Technische Universit{\"a}t M{\"u}nchen, Boltzmannstr.\ 3, 85748
Garching, Germany, {\tt schmidt@ma.tum.de}}

\addtocounter{footnote}{1}\footnotetext{Correspondence to: M. Ortiz,
Graduate Aerospace Laboratories, California Institute of Technology,
Pasadena, CA 91125, United States, {\tt ortiz@aero.caltech.edu}}

\maketitle

\begin{abstract}
This work is concerned with the formulation of a general framework
for the analysis of meshfree approximation schemes and with the
convergence analysis of the Local Maximum-Entropy (LME) scheme as a
particular example. We provide conditions for the convergence in
Sobolev spaces of schemes that are {\sl $n$-consistent}, in the
sense of exactly reproducing polynomials of degree less or equal to
$n\geq 1$, and whose basis functions are of {\sl rapid decay}. The
convergence of the LME in $W^{1, p}_{\rm loc}(\Omega)$ follows as a
direct application of the general theory. The analysis shows that
the convergence order is linear in $h$, a measure of the {\it
density} of the point set. The analysis also shows how to
parameterize the LME scheme for optimal convergence. Because of the
convex approximation property of LME, its behavior near the boundary
is singular and requires additional analysis. For the particular
case of polyhedral domains we show that, away from a small singular
part of the boundary, any Sobolev function can be approximated by
means of the LME scheme. With the aid of a capacity argument, we
further obtain approximation results with truncated LME basis
functions in $H^1(\Omega)$ and for spatial dimension $d > 2$.

\end{abstract}

\begin{section}{Introduction}

{\sl Meshfree approximation schemes} (cf., e.~g.,
\cite{meshfree_review} for a review) are advantageous in a number of
areas of application, e.~g., those involving Lagrangian descriptions
of unconstrained flows (cf., e.~g., \cite{LiHabbalOrtiz2010} for a
representative example) where methods based on triangulation, such
as the finite-element method, inevitably suffer from problems of
mesh-entanglement. The present work is concerned with the
formulation of a general framework for the analysis of meshfree
approximation schemes (cf., e.~g., \cite{meshfree_convergence} for
representative past work) and with its application to the Local
Maximum-Entropy (LME) scheme as an example. By way of conceptual
backdrop, we may specifically envision time-independent problems for
which the solutions of interest follow as the minimizers of a
functional $F\colon X \to \bar{\mathbb{R}}$, where $X$ is a
topological vector space. General conditions for the existence of
solutions are provided by the Tonelli's theorem (e.~g.,
\cite{dalMaso1993}). In this framework, an approximation scheme is a
sequence $X_k$ of subspaces of $X$, typically of finite dimension,
defining a corresponding sequence of {\sl Galerkin reductions} of
$F$,
\begin{equation}
    F_k(u)
    =
    \left\{
    \begin{array}{ll}
        F(u), & \text{if } u\in X_k, \\
        +\infty, & \text{otherwise} .
    \end{array}
    \right.
\end{equation}
An approximation scheme is then said to be {\sl convergent} if it
has the following density property: For every $u\in X$, there exists
a sequence $u_k\in X_k$ such that $\lim_{k\to\infty} u_k = u$. The
connection between density of the approximation scheme and
convergence is provided by the following proposition
\cite{ContiHauretOrtiz2007}.

\begin{proposition}\label{prop:abstract}
Let $X$ be endowed with two metrizable topologies $S$ and $T$, with
$T$ finer than $S$. Let $F\colon X \to \bar{\mathbb{R}}$ be coercive
in $(X,S)$ and continuous in $(X,T)$. Let $X_k$ be a dense sequence
of sets in $(X,T)$ and let $F_k$ be the corresponding sequence of
Galerkin reductions of $F$. Then the sequence $F_k$
$\Gamma$-converges to the lower semicontinuous envelope of $F$ and
is equicoercive in $(X,S)$.
\end{proposition}

We recall that $\Gamma$-convergence is a powerful notion of
variational convergence of functionals that, in particular, implies
convergence of minimizers. Thus, if the sequence $F_k$ is
equicoercive, then the minimizers of $F$ are accumulation points of
minimizers of $F_k$, i.~e., if $F_k(u_k)=\inf F_k$ then the sequence
$u_k$ has a subsequence that converges to a minimizer of $F$. We
also recall that the topology $T$ is finer than $S$, i.~e., any
converging sequence for $T$ converges for $S$. In applications, $T$
is typically a metric or normed topology and $S$ the corresponding
weak topology.

It thus follows that, within the general framework envisioned here,
the analysis of convergence of approximation schemes reduces to
ascertaining the density property. Towards this end, in
Section~\ref{sec:: general convergence analysis} we begin by
analyzing meshfree approximation schemes that are {\sl
$n$-consistent}, in the sense of exactly reproducing polynomials of
degree less or equal to $n\geq 1$, and whose basis functions are of
{\sl rapid decay}. Specifically, for schemes subordinate to point
sets possessing a certain geometrical regularity property that we
term {\sl $h$-density}, we prove a uniform error bound for
consistent and rapidly-decaying approximation schemes. In addition,
we show that the sets of functions spanned by consistent and
rapidly-decaying approximation schemes are dense in Sobolev spaces.

In Sections \ref{sec:: convergence analysis} and \ref{sec::
boundary-density}, we apply the general results of
Section~\ref{sec:: general convergence analysis} to the Local
Maximum-Entropy (LME) approximation scheme of Arroyo and
Ortiz~\cite{arroyo_ortiz_06} (see also \cite{arroyo_ortiz_06_b,
sukumar_07, cyron_09, gonzalez_10}). The LME scheme has been
extensively assessed numerically over a broad range of test problems
\cite{arroyo_ortiz_06, quak_09, LiHabbalOrtiz2010}, but a rigorous
convergence analysis has been heretofore unavailable. The general
theory of Section~\ref{sec:: general convergence analysis} readily
establishes the density of the LME approximation spaces $X_k$ in
$W^{1, p}_{\rm loc}(\Omega)$, cf.~Section \ref{sec:: convergence
analysis}. In particular, the analysis shows that the convergence
order is linear in $h$, a measure of the {\it density} of the point
set. These convergence rates and the corresponding error bounds are
in agreement with the numerical results reported
in~\cite{arroyo_ortiz_06}, and are comparable to those of the
first-order finite element method (cf., e.~g., \cite{ciarlet}).
Conveniently, the analysis also shows how to choose the LME {\sl
temperature parameter} so as to obtain optimal convergence. This
optimal choice is in agreement with that determined in
\cite{arroyo_ortiz_06, LiHabbalOrtiz2010} by means of numerical
testing.

The LME scheme is a {\sl convex approximation scheme} in which the
basis functions are constrained to take non-negative values. By
virtue of this restriction, the LME scheme is defined for convex
domains only. Consequently, its behavior near the boundary is
somewhat singular and requires careful additional analysis. In
Section \ref{sec:: boundary-density}, for the particular case of
polyhedral domains we show that, away from a small singular part of
the boundary, any Sobolev function can indeed be approximated by
means of the LME scheme. Then, with the aid of a capacity argument
we obtain approximation results with truncated LME basis functions
in $H^1(\Omega)$ and for spatial dimension $d > 2$.

\end{section}

\begin{section}{Prolegomena}
\label{sec:: background}

The open $d$-ball $B({x}, \delta)$ of radius $\delta$ centered at $x$ is the set $\{{y} \in \mathbb{R}^d\colon |{y}- {x}|< \delta\}$. The closed $d$-ball $\bar{B}({x}, \delta)$ of radius $\delta$ centered at $x$ is the set $\{{y} \in \mathbb{R}^d\colon |{y}- {x}|\leq \delta\}$. Given a set $A \subset \mathbb{R}^d$, we denote by $\bar{A}$ its closure, and by $\partial A$ its boundary. By a {\sl domain} we shall specifically understand an open and bounded subset of $\mathbb{R}^d$. Given a {\sl point set} $P \in (\mathbb{R}^d)^N$, we denote by $\overline{\conv}(P)$ its closed convex hull~\cite{rockafellar_70_a}, and by $\conv(P)$ the interior of $\overline{\conv}(P)$. We recall that a {\sl $d$-simplex} $T \subset \mathbb{R}^d$ is the convex hull of $d+1$ affinely independent points~\cite{rockafellar_70_a}. Given a bounded set $A \subset \mathbb{R}^d$, its {\sl size} $h_A$ is the diameter of the smallest ball containing $A$.

The following definitions formalize the notion of a point set $P \subset \Omega$ that approximates a domain $\Omega$ uniformly.

\begin{definition}[$h$-covering]\label{def:: h-covering}
We say that a point set $P \subset \mathbb{R}^d$ is an {\it h}-\emph{covering} of a set $A \subset \mathbb{R}^d$, $h>0$, if for every ${x}\in A$ there exists a $d$-simplex $T_{{x}}$ of size $h_{T_{{x}}} < h$ and with vertices in $P$ such that ${x}\in T_{{x}}$.
\end{definition}
\begin{definition}[$h$-density]\label{def:: h-density}
We say that a point set $P \subset \mathbb{R}^d$ has {\it h}-\emph{density} bounded by $\tau > 0$ if for every $x \in \mathbb{R}^d$, $\# \big(P \cap {\bar B}({x}, h)\big) \leq \tau$.
\end{definition}

For a point set $P \subset \Omega$ with {\it h}-density bounded by $\tau$, the following proposition bounds its number of points in rings of $\mathbb{R}^d$.

\begin{proposition}\label{prop:: UB number node points on rings}
Assume $P \subset \Omega$ has {\it h}-density bounded by $\tau$, for some $h$, $\tau >0$. Then there is a constant $c > 0$ that depends on $\tau$ and $d$ such that,
\begin{equation}
    \# \left( P \cap \left({\bar B}({x}, th) \setminus B(x, (t-1)h) \right) \right)
    \leq
    c\, t^{d-1},
\end{equation}
$\forall {x} \in \Omega$ and integers $t \geq 1$.
\end{proposition}
\begin{proof}
Let
\begin{equation}
\begin{split}
    E_1
    &= \left\{y \in \R^d \colon {\rm dist}\left(y, {\bar B}({x}, th) \setminus B(x, (t-1)h)\right) \le h \right\} \\
    &= {\bar B}({x}, (t+1)h) \setminus B(x, (t-2)h)
\end{split}
\end{equation}
and
\begin{equation}
\begin{split}
    E_2
    &= \left\{y \in \R^d \colon {\rm dist}\left(y, {\bar B}({x}, th) \setminus B(x, (t-1)h)\right) \le 2 h \right\} \\
    &= {\bar B}({x}, (t+2)h) \setminus B(x, (t-3)h).
\end{split}
\end{equation}
Then for every $y \in {\bar B}({x}, th) \setminus B(x, (t-1)h)$ there is a $z \in Z := E_1 \cap h d^{-\frac{1}{2}} \mathbb{Z}^d$ such that $|y - z| \le h$ and so
\begin{equation}
    {\bar B}({x}, th) \setminus B(x, (t-1)h)
    \subset
    \bigcup_{z \in Z} \bar{B}(z, h).
\end{equation}
On the other hand, $z + [0, h d^{-\frac{1}{2}})^d \subset E_2$ for all $z \in Z$, and thus
\begin{equation}
    \# Z h^d d^{-\frac{d}{2}}
    =
    \left| \bigcup_{z \in Z} z + [0, h d^{-\frac{1}{2}})^d \right|
    \le |E_2|
    \le |B(0,1)| h^d \left( (t+2)^d - (t-3)^d \right)
    \le c h^d t^{d-1}.
\end{equation}
Consequently,
\begin{align}
  \# \left( P \cap \left({\bar B}({x}, th) \setminus B(x, (t-1)h) \right) \right)
  \le \# \left( P \cap \bigcup_{z \in Z} \bar{B}(z, h) \right)
  \le c d^{\frac{d}{2}} \tau t^{d-1}.
\end{align}
\end{proof}
\end{section}

\begin{section}{Convergence Analysis of General Meshfree Approximation Schemes}
\label{sec:: general convergence analysis}

In this Section we analyze meshfree approximation schemes that are $n$-consistent and whose shape functions are of {\sl rapid decay}. Specifically, we prove a uniform error bound for consistent and rapidly-decaying approximation schemes. In addition, we show that the set of functions spanned by consistent and rapidly-decaying approximation schemes are dense in Sobolev spaces.

Let $\Omega \subset \mathbb{R}^d$ be a domain. By an approximation scheme $\{I,W,P\}$ we mean a collection $W = \{w_a,\ a \in I\}$ of shape functions and a point set $P$, both indexed by $I$. Given an approximation scheme $\{I,W,P\}$, we approximate functions $u\colon\Omega\to\mathbb{R}$ by functions in the span $X$ of $W$ of the form
\begin{equation}\label{def:: uh_}
    u_I({x}) = \sum_{a \in I} u(x_a) w_a({x}),
\end{equation}
provided that this operation is well defined. More generally, we shall consider sequences of approximation schemes $\{I_k,W_k,P_k\}$ and let
\begin{equation}\label{def:: uh}
    u_k({x}) = \sum_{a \in I_k} u(x_a) w_a({x}),
\end{equation}
be the corresponding sequence of approximations to $u$ in the sequence $X_k$ of finite-dimensional spaces of functions spanned by $W_k$. We note that, for simplicity, we assume that all functions are defined over a common domain $\Omega$. Depending on the approximation scheme, this assumption may implicitly restrict the type of domains that may be considered, e.~g., polyhedral domains. The aim then is to ascertain conditions on the approximation scheme under which $u_k \to u$ in an appropriate Sobolev space $W^{m,p}(\Omega)$.

We recall the following definition of consistency of approximation schemes~\cite{strang_73}.
\begin{definition}[Consistency]
We say that an approximation scheme $\{I,W,P\}$ is consistent of order $n \geq 0$, or $n$-consistent, relative to a point set $P$ if it exactly interpolates polynomials of degree less or equal to $n$ within $\Omega$, i.~e., if
\begin{equation}\label{eq:Cons}
    x^\alpha = \sum_{a\in I} x_a^\alpha w_a(x)
\end{equation}
for all multiindices $\alpha$ of degree $|\alpha| \leq n$.
\end{definition}
A simple binomial expansion shows that~\eqref{eq:Cons} can equivalently be replaced by
\begin{subequations}
\label{eq:Cons2}
\begin{align}
    \sum_{a\in I} w_a(x) &= 1,\\
    \sum_{a\in I} w_a(x)(x_a - x)^\alpha  &= 0,\ \forall \alpha \in \mathbb{N}^d,\, 0 < |\alpha|
    \leq n,
\end{align}
\end{subequations}
in the definition of consistency.

Consistency results in a number of identities involving the partial derivatives of the shape functions, which we record next for subsequent use (cf.~\cite{ciarlet_72}).

\begin{lemma}\label{lem:DCons}
Let $\{I,W,P\}$ be an approximation scheme. Suppose that $W$ consists of $C^r(\Omega)$ shape-functions that are $n$th-order consistent relative to $P$ in $\Omega$. Let $\alpha$, $\beta$ be multiindices, with $0 \leq |\alpha|\leq n$, $0 \leq |\beta| \leq r$. Then,
\begin{equation}\label{eq:DCons}
    \sum_{a\in I}
    D^\beta w_{a}(x) (x_{a} - x)^{\alpha}
    =
\begin{cases}
    \alpha!, & \text{if } \alpha = \beta, \\
    0, & \text{otherwise}.
\end{cases}
\end{equation}
\end{lemma}

\begin{proof}
We proceed by induction on $|\beta|$. For $\beta = 0$, the identity \eqref{eq:DCons} follows directly from consistency. Let $0 \leq m < r$ and assume eq.~\eqref{eq:DCons} holds for all multiindices $\alpha$, $\beta$ such that $0\leq |\beta| \leq m$ and $0 \leq |\alpha|\leq n$. Let $\beta$ and $\gamma$ be such that $|\beta|=m$ and $|\gamma|=1$. Then, for $\alpha=0$ we have, by consistency,
\begin{equation}
    \sum_{a\in I} D^{\beta+\gamma} w_a(x)
    =
    D^{\beta+\gamma} \sum_{a\in I} w_a(x)
    =
    0,
\end{equation}
whereas for $0 < |\alpha|\leq n$ we have, also by consistency,
\begin{equation}\label{eq:Dcons2}
\begin{split}
    &
    \sum_{a\in I}
    D^{\beta+\gamma} w_a(x) (x_{a} - x)^{\alpha}
    \\ & =
    D^\gamma
    \sum_{a\in I}
    D^\beta w_a(x)(x_{a} - x)^{\alpha}
    +
    \sum_{a\in I}
    D^\beta w_a(x) (\alpha\cdot\gamma) (x_{a} - x)^{\alpha - \gamma}
    \\ & =
    (\alpha\cdot\gamma) \sum_{a\in I}
    D^\beta w_a(x) (x_{a} - x)^{\alpha - \gamma}
\end{split}
\end{equation}
Suppose that $\alpha \neq \beta+\gamma$. Then, from
\eqref{eq:DCons},
\begin{equation}
    \sum_{a\in I} D^\beta w_a(x) (x_{a} - x)^{\alpha - \gamma}
    =
    0.
\end{equation}
Suppose, contrariwise, that $\alpha = \beta+\gamma$. Then, also from
\eqref{eq:DCons},
\begin{equation}
    (\alpha\cdot\gamma) \sum_{a\in I}
    D^\beta w_a(x) (x_{a} - x)^{\alpha - \gamma}
    =
    (\alpha\cdot\gamma) (\alpha - \gamma)!
    =
    \alpha!,
\end{equation}
whereupon \eqref{eq:Dcons2} becomes
\begin{equation}
    \sum_{a\in I}
    D^{\beta+\gamma} w_a(x) (x_{a} - x)^{\alpha}
    =
    \alpha!,
\end{equation}
and \eqref{eq:DCons} holds for all multiindices $\beta$ of degree
$m+1$.
\end{proof}

We recall that the Taylor approximation of order $r$ of a function
$u\in C^{r+1}(\Omega)$ at $y\in\Omega$ is
\begin{equation}
    T_{r}(u)(x,y)
    =
    \sum_{|\alpha| \leq r}
    \frac{1}{\alpha!} D^\alpha u(x) (y - x)^{\alpha}
\end{equation}
and its remainder is
\begin{equation}
    R_{r+1}(x,y)
    =
    u(y) - T_{r}(u)(x,y),
\end{equation}
which turns out to be
\begin{equation}\label{eq:TaylorResidual}
    R_{r+1}(x,y)
    =
    \sum_{|\alpha| = r + 1}
    \frac{1}{\alpha!} D^\alpha u\big(x + \theta (y - x)\big) (y -
    x)^{\alpha},
\end{equation}
for some $\theta \in (0,1)$.

Functions in the span of a consistent shape-function basis satisfy the following {\sl multipoint Taylor formula} (cf.~\cite{ciarlet, ciarlet_72}).

\begin{proposition}[Multipoint Taylor formula]\label{prop:Taylor}
Let $W$ be a $C^r(\Omega)$ shape-function set $n$th-order consistent
relative to a point set $P$ in $\Omega$, $u\in
C^{\ell+1}(\conv(\Omega))$ and $m = \min\{n, \ell\}$. Then,
\begin{equation}\label{eq:Duh}
    D^\alpha u_I(x)
    =
    D^\alpha u(x)
    +
    \sum_{a\in I} R_{m+1}(x_a,x) D^\alpha w_a(x).
\end{equation}
for all $|\alpha|\leq \min\{m,r\}$ and $x\in\Omega$.
\end{proposition}
\begin{proof}
The proof follows that of Theorem~1 of~\cite{ciarlet_72}. From the Taylor expansion of order $m$ of $u$ at $x$ we have
\begin{equation}
    u(x_a) =
    \sum_{|\beta|\leq m}
    \frac{1}{\beta !} D^\beta u(x) (x_a - x)^{\beta}
    +
    R_{m + 1}(x, x_a).
\end{equation}
whence it follows that
\begin{equation}
\begin{split}
    D^\alpha u_I(x)
    & =
    \sum_{a\in I} u(x_a) D^\alpha w_a(x)
    \\ & =
    \sum_{a\in I} \left(
    \sum_{|\beta|\leq m}\frac{1}{\beta!} D^\beta u(x) (x_a - x)^{\beta} +
    R_{m+1}(x, x_a) \right) D^\alpha w_a(x) 
    \\ & =
    \sum_{|\beta|\leq m}\frac{1}{\beta!} D^\beta u(x)
    \left(\sum_{a\in I} D^\alpha w_a(x)(x_a - x)^{\beta}\right) + \sum_{a\in I}
    R_{m+1}(x, x_a) D^\alpha w_a(x),
\end{split}
\end{equation}
and \eqref{eq:Duh} follows from Lemma~\ref{lem:DCons}.
\end{proof}

We recall that a function $f \in C^\infty(\mathbb{R}^d)$ is said to be {\sl rapidly decreasing} if~\cite{rudin_91}
\begin{equation}\label{def:: rapidly decreasing function}
    \sup_{|\alpha|\leq N} \sup_{x\in\mathbb{R}^d}
    (1+|x|^2)^N |(D^\alpha f)(x)| < \infty
\end{equation}
for all $N=0$, $1$, $2$, $\dots$, where $|x|^2=\sum x_i^2$.

The next definition formalizes a polynomial-decay condition of the shape functions and their derivatives.

\begin{definition}[Approximation scheme with polynomial decay]
\label{def:: polynomial decay shape functions}
We say that an approximation scheme $\{I,W,P\}$ has a \emph{polynomial decay of order} $(r,s)$ for constants $c > 0$ and $h>0$ if the basis $W$ is in $C^r(\Omega)$, and
\begin{equation}\label{def:: eq::poly decay}
    \sup_{|\alpha|\leq r}
    \sup_{x\in\Omega}
    \sup_{a\in I}
    \left(1 + \left|\frac{x - x_a}{h} \right|^2\right)^s {h}^{|\alpha|}|D^\alpha w_a(x)| <
    c.
\end{equation}
A sequence of approximation schemes $\{I_k,W_k,P_k\}$ has a \emph{uniform polynomial decay of order} $(r,s)$ if there exists a constant $c >0$ and a sequence $h_k \to 0$ such that, for each $k$, $\{I_k,W_k,P_k\}$ has a polynomial decay of order $(r,s)$ for constants $c$ and $h_k$.
\end{definition}
We note that, if the shape functions are invariant under a linear transformation, then the left hand side of~\eqref{def:: eq::poly decay} is also invariant under the same transformation change.

The next proposition establishes a key concentration property of approximation schemes with polynomial decay.

\begin{proposition}[Shape-function concentration]
\label{prop:: sharp concentration bound}
Let $\{I, W, P\}$ be an approximation scheme. Suppose that there exists $\tau > 0$ such that $P$ has $h$-density bounded by $\tau$. Suppose, in addition, that the approximation scheme has polynomial decay of order $(r,s)$ for constants $c>0$ and $h$, with $2s > d$. Then, for every $\theta > 0$ there exists a constant $c_{\theta} > 0$ such that
\begin{equation}
    \sum_{{x}_a \in P \setminus {\bar B}({x}, c_{\theta}h)} |w_a({x})|
    \leq
    \theta,
\end{equation}
everywhere in $\Omega$.
\end{proposition}
\begin{proof}
For every nonnegative integer $t \geq 1$, let $U_t({x})$ be the ring of node points $U_t({x}) = \{{x_a} \in P\colon (t-1) h \leq |{x_a} - {x}| < t h \}$. Note that $P = \cup_{t=1}^{\infty} U_t(x)$. By Proposition~\ref{prop:: UB number node points on rings}, there exists a constant $c^{\prime}$ that depends on $\tau$ and $d$ such that, for any $t \geq 1$, the number of node points of $U_t(x)$ is at most $\# U_t({x}) \leq c^{\prime} t^{d-1}$. Since the approximation scheme has polynomial decay of order $(r,s)$ with $2s > d$, for any integer $c_{\theta} \geq 1$ we have
\begin{equation}
\begin{split}
    \sum_{{x}_a \in P \setminus {\bar B}({x}, c_{\theta}h)} |w_a({x})|
    & \le
    \sum_{t = c_{\theta}}^{\infty} \sum_{{x}_a
    \in U_t({x})} |w_a({x})|
    \leq
    \sum_{t = c_{\theta}}^{\infty} c^{\prime} t^{d-1} c
    ((t-1)^2+1)^{-s}
    \leq
    \sum_{t = c_{\theta}}^{\infty} c^{\prime} t^{d-1} c t^{-2s}
    \\ & \leq
    \sum_{t = c_{\theta}}^{\infty} c^{\prime} c t^{-1 - (2s - d)}.
\end{split}
\end{equation}
Note that the series $\sum_{t = 1}^{\infty} c^{\prime} c t^{-1 - (2s - d)}$ is finite. In particular, there exists a value $c_{\theta} < \infty$, depending on $d$, $\tau$, and $\theta$, such that $\sum_{t = c_{\theta}}^{\infty} c^{\prime} c t^{-1 - (2s - d)} \leq \theta$.
\end{proof}

For an $n$-consistent approximation scheme with sufficiently high polynomial decay, the following theorem provides a uniform interpolation error bound.

\begin{theorem}[Uniform interpolation error bound]\label{thm:: uniform error bound}
Let $\{I,W,P\}$ be an approximation scheme. Suppose that:
\begin{itemize}
    \item[i)] The approximation scheme is $n$-consistent, $n\geq 0$.
    \item[ii)] There exists $\tau > 0$ such that $P$ has $h$-density bounded by $\tau$.
    \item[iii)] The approximation scheme has polynomial decay of order $(r,s)$ with $2s > d + m + 1$, where ${m}= \min \{n,\ell\}$.
\end{itemize}
Let $u \in C^{\ell+1}(\overline{\conv}(\Omega))$. Then, there exists a constant $C < \infty$ such that
\begin{equation}
    \left|D^{\alpha}u_I({x}) - D^{\alpha}u({x}) \right|
    \leq
    C
    \left\|D^{m+1}u\right\|_{\infty} h^{m + 1 - |\alpha|},
\end{equation}
for every $|\alpha| \leq \min \{m, r\}$ and ${x} \in \Omega$.
\end{theorem}
\begin{proof}
By Proposition~\ref{prop:Taylor},
\begin{equation}\label{eq:: thm1}
    \left|D^{\alpha}u_I({x}) - D^{\alpha}u({x}) \right|
    \leq
    \sum_{a \in I} |R_{m+1}({x}_a, {x})|
    \left|D^{\alpha} w_a({x})\right|
\end{equation}
for every multiindex $\alpha$ of degree less or equal to $\min\{m, r\}$ and every ${x} \in \Omega$. Next, we proceed to bound the right-hand side of this inequality. For each nonnegative integer $t \geq 1$, let $U_t({x})$ be the ring of nodal points $U_t({x}) = \{{x_a} \in P\colon (t-1)h \leq |{x_a} - {x}| < th\}$. Note that $P = \cup_{t=1}^{\infty} U_t(x)$. By Proposition~\ref{prop:: UB number node points on rings}, there exists a constant $c$ that depends on $\tau$ and $d$ such that, for any $t \geq 1$, the number of node points of $U_t(x)$ is at most $\# U_t({x}) \leq c t^{d-1}$. In addition, from~\eqref{eq:TaylorResidual} we have
\begin{equation}
    |R_{m+1}({x_a}, {x})|
    \leq
    \frac{d^{m+1}}{(m+1)!}
    \left\|D^{m+1}u\right\|_{\infty} (th)^{m+1}.
\end{equation}
By the assumption of polynomial decay there exists a constant $0 < c^{\prime} < \infty$ such that
\begin{equation}
    \left|D^{\alpha}w_a({x})\right|
    \leq
    c^{\prime} \left(\left|\frac{x - x_a}{h}\right|^2 + 1 \right)^{-s}h^{-|\alpha|}
    \leq
    c^{\prime} \left((t-1)^2 + 1 \right)^{-s}h^{-|\alpha|}
    \leq
    5 c^{\prime} t^{-2s}h^{-|\alpha|},
\end{equation}
for every $x_a \in U_t(x)$. From the preceding bounds we have
\begin{equation}
\begin{split}
    &
    \sum_{a \in I} |R_{m+1}({x}_a, {x})|
    \left|D^{\alpha}w_a({x})\right|
    \\ & =
    \sum_{t = 1}^{\infty} \sum_{{x}_a \in U_t({x})}
    |R_{m+1}({x}_a, {x})| \left|D^{\alpha}w_a({x})\right|
    \\ & \leq
    \sum_{t = 1}^{\infty} \sum_{{x}_a \in U_t({x})}
    \frac{d^{m+1}}{(m+1)!}\left\|
    D^{m+1}u\right\|_{\infty} (th)^{m+1} 5 c^{\prime} t^{-2s}h^{-|\alpha|}
    \\ & \leq
    5 c^{\prime} \frac{d^{m+1}}{(m+1)!} \left\|
    D^{m+1}u\right\|_{\infty} h^{m+1 - |\alpha|}
    \sum_{t = 1}^{\infty} \# U_t({x}) \, t^{m + 1 - 2s}
    \\ & \leq
    5 c^{\prime} \frac{d^{m+1}}{(m+1)!}\left\| D^{m+1}u\right\|_{\infty}
    h^{m+1 - |\alpha|} \sum_{t = 1}^{\infty} c t^{d-1} t^{m + 1 - 2s}
    \\ & \leq
    5 c^{\prime} c \frac{d^{m+1}}{(m+1)!}\left\|
    D^{m+1}u\right\|_{\infty} h^{m + 1 - |\alpha|}
    \sum_{t = 1}^{\infty} t^{d + m - 2s}.
\end{split}
\end{equation}
Since $d + m - 2s < -1$, it follows that $\sum_{t=1}^{\infty}
t^{d + m + 1 - 2s} < \infty$. Thus,
\begin{equation}
    \left|D^{\alpha} u_I({x}) - D^{\alpha} u({x}) \right|
    \leq
    C \left\|D^{m+1}u\right\|_{\infty} h^{m+1 - |\alpha|},
\end{equation}
for every $x \in \Omega$, where we note that the constant $C = 5 c^{\prime} c \frac{d^{m+1}}{(m+1)!} \left(\sum_{t=1}^{\infty} t^{d+m-2s}\right)$ depends on $\tau$, $d$, $c^{\prime}$, and $s$.
\end{proof}

The following corollaries to Theorem~\ref{thm:: uniform error bound} show that a function in $W^{m,p}(\Omega)$ can be approximated by means of consistent approximation schemes of polynomial decay.

\begin{corollary}
\label{coro:: approx in Sobolev space for smooth functions} Under
the assumptions of Theorem~\ref{thm:: uniform error bound},
\begin{equation}
    \left\|u_I - u \right\|_{W^{j,p}(\Omega)}
    \leq
    C \|D^{m+1}u\|_{\infty} h^{1 + m - j},
\end{equation}
for $ 1\leq p < \infty$, $j = \min\{n, r, \ell\}$ and every $u \in C^{\ell+1} (\overline{\conv}(\Omega))$.
\end{corollary}
\begin{proof}
By Theorem~\ref{thm:: uniform error bound}, there exists a constant
$0 < C < \infty$ such that,
\begin{equation}
    \left\|u_I - u \right\|_{C^j(\bar{\Omega})}
    \leq
    C \|D^{m+1}u\|_{\infty} h^{1 + m - j},
\end{equation}
so that the assertion follows from the continuous embedding
$C^j(\bar{\Omega}) \hookrightarrow W^{j,p}(\Omega)$.
\end{proof}

Convergence in $W^{j,p}(\Omega)$ finally follows from standard theory of approximation by continuous functions (cf.~e.~g., \cite{adams_03}). For completeness, we proceed to note a particular case of practical relevance. We recall that a domain $\Omega$ satisfies the {\sl segment condition} if, for all $x$ in the boundary of $\Omega$, there exists a neighborhood $U_x$ and a direction $y_x \neq 0$ such that, for any point $z \in \bar{\Omega} \cap U_x$, the point $z + t y_x$ belongs to $\Omega$, for all $0<t<1$ (\cite{adams_03}, \S 3.21). Convex domains satisfy the segment condition without additional restrictions on their boundaries. We additionally recall (\cite{adams_03}, thm.~3.22) that, if $\Omega$ satisfies the segment condition, then the functions of $C^{\infty}_c (\mathbb{R}^d)$ restricted to $\Omega$ are dense in $W^{j,p}(\Omega)$ for $1 \leq p < \infty$.

\begin{corollary}\label{coro:: density in Sobolev space}
Let $\Omega$ be a domain satisfying the segment condition.  Suppose that the assumptions of Theorem~\ref{thm:: uniform error bound} hold uniformly for $h_k\to 0$. Let $ 1\leq p < \infty$ and $j = \min\{n, r\}$. Then, for every $u \in W^{j,p}(\Omega)$ there exists a sequence $u_k \in X_k$ such that $u_k \to u$.
\end{corollary}
\begin{proof}
By the density of $C^{\infty}_c (\mathbb{R}^d)$ in $W^{j,p}(\Omega)$, there is a sequence of functions $v_i\in C^\infty_c(\mathbb{R}^d)$ whose restrictions to $\Omega$ converge to $u $ in $W^{j,p}(\Omega)$. The corollary then follows by approximating each $v_i$ by a sequence $u_{i_k} \in X_k$ and passing to a diagonal sequence.
\end{proof}
\end{section}

\begin{section}{Application to the Local Maximum-Entropy (LME) Approximation Scheme: Interior Estimates} \label{sec:: convergence analysis}

In this Section, we specialize the results of Section~\ref{sec:: general convergence analysis} to the LME approximation schemes.  We begin with a brief review of the definition and some of the properties of the Local Max-Ent Approximation scheme of Arroyo and Ortiz~\cite{arroyo_ortiz_06} (see also~\cite{arroyo_ortiz_06_b, sukumar_07} for a description of the method, and~\cite{sukumar_04, sukumar_05, sukumar_07_b} for related work). We recall that a {\sl convex approximation scheme} is a first-order consistent approximation scheme $\{I,W,P\}$ whose shape functions are non-negative. Convex approximation schemes satisfy a weak Kronecker-delta property at the boundary (cf.~\cite{arroyo_ortiz_06}), i.~e., the approximation on the boundary of the domain does not depend on the nodal data over the interior points. This property simplifies the enforcement of essential boundary conditions. As pointed out in~\cite{arroyo_ortiz_06}, in a convex approximation scheme the shape functions ${w}_a({x})$, $a\in I$, are well-defined if and only if ${x} \in \overline{\conv}(P)$. Therefore, for such schemes to be feasible the domain $\Omega$ must be a subset of $\overline{\conv}(P)$.

The Local Maximum-Entropy (LME) approximation scheme \cite{arroyo_ortiz_06} is a convex approximation scheme that aims to satisfy two objectives simultaneously:
\begin{enumerate}
    \item {\it Unbiased statistical inference} based on the nodal data.
    \item Shape functions of {\it least width}.
\end{enumerate}
Since for each point $x$, the shape functions of a convex approximation scheme are nonnegative and add up to $1$, they can be thought of as the probability distribution of a random variable. The statistical inference of the shape functions is then measured by the {\it entropy} of the associated probability distribution, as defined in information theory~\cite{shannon_48_a, jaynes_57_a, khinchin_57_a}. The entropy of a probability distribution $p$ over $I$ is:
\begin{equation}\label{eq:: entropy functional}
    H(p) = -\sum_{a \in I} p_a \log p_a,
\end{equation}
where $0 \log 0 = 0$. The least biased probability distribution $p$ is that which maximizes the entropy. In addition, the {\sl width} of a non-negative function $w$ about a point $\xi$ is identified with the second moment
\begin{equation}
\label{eq:: Euclidean width}
    U_\xi(w) = \int_{\Omega}w({x})|{x} - \xi|^2 \, dx.
\end{equation}
Thus, the width $U_\xi(w)$ measures how concentrated $w$ is about $\xi$. According to this measure of width, the most local approximation scheme is that which minimizes the total width
\begin{equation}
    \label{eq:: Rajan functional}
    U(W)
    =
    \sum_{a \in I} U_a(w_a)
    =
    \int_{\Omega} \sum_{a \in I}
    w_a({x})|{x} - {x}_a|^2 \, dx.
\end{equation}
The Local Maximum-Entropy approximation schemes combine the functionals~\eqref{eq:: entropy functional} and~\eqref{eq:: Rajan functional} into a single objective. More precisely, for a parameter $\beta >0$, the LME approximation scheme is the minimizer of the functional
\begin{equation}
    F_\beta(W)
    =
    \beta U(W) - H(W)
\end{equation}
under the restriction of first-order consistency. Because of the local nature of this functional, it can be minimized pointwise, leading to the local convex minimization problem:
\begin{equation}\nonumber
\left.
\begin{aligned}
    &
    \min f_{\beta}({x}, {w}({x}))
    =
    \sum_{a \in I} w_a({x})
    |{x} - {x}_a|^2
    +
    \frac{1}{\beta}
    \sum_{a \in I} w_a({x}) \log w_a({x}),
    \\ &
    \mbox{subject to: }
    {w_a}({x}) \geq 0,\ a \in I, \quad
    \sum_{a \in I} w_a({x}) = 1, \quad
    \sum_{a \in I} w_a({x}) {x}_a = {x}.
\end{aligned}
\right\} \qquad\qquad \text{(LME)}
\end{equation}
In the limit of $\beta\to\infty$ the function $f_{\beta}$ reduces to the power function of Rajan~\cite{rajan_94_a}, whose minimizers define the piecewise-affine shape functions supported by the Delaunay triangulations associated with $P$.

Next we collect alternative characterizations of the LME shape functions based on duality theory. Let $Z\colon \mathbb{R}^d \times \mathbb{R}^d \to \mathbb{R}$ be the {\sl partition function}
\begin{equation} \label{def:: partition function}
    Z({x}, {\lambda})
    =
    \sum_{a \in I}
    e^{-\beta|{x} - {x}_a|^2
    +
    \langle{\lambda}({x}), {x} - {x}_a\rangle}
\end{equation}
of the point set. For every point ${x} \in \overline{\conv}(P)$, the problem (LME) has a unique solution $\{w_a^*({x})\colon  a \in I\}$. Moreover, for every point ${x} \in \conv(P)$, the optimal shape functions $w^*_a({x})$ at ${x}$ are of the form
\begin{equation}\label{eq:: optimal probas}
    w^*_a({x})
    =
    \frac{e^{
    -\beta|{x} - {x}_a|^2
    +
    \langle{{{\lambda}}^*}({x}), {x} - {x}_a\rangle}}
    {\sum_{b \in I}
    e^{-\beta|{x} - {x}_b|^2
    +
    \langle{{{\lambda}}^*}({x}), {x} - {x}_b\rangle}},
\end{equation}
where the vector ${\lambda}^*({x}) \in \mathbb{R}^d$ minimizes the function
\begin{equation} \label{eq:: optimal lambda}
    \log Z({x}, {\lambda})
    =
    \log \left(
    \sum_{a \in I}
    e^{-\beta|{x} - {x}_a|^2
    +
    \langle{\lambda},{x} - {x}_a\rangle}\right).
\end{equation}
At points ${x}$ belonging to the boundary of $\conv {P}$, the shape functions take expressions similar to~\eqref{eq:: optimal probas} that solely involve the node points on the minimal face of $\overline{\conv}(P)$ that contains ${x}$. The gradient of $\log Z({x}, {\lambda})$ with respect to ${\lambda}$ is
\begin{equation}
    {r}({x}, {\lambda})
    \equiv
    \frac{\partial}{\partial {\lambda}} \log Z({x},
    {\lambda})
    =
    \sum_{a \in I} w_a({x},
    {\lambda}) ({x} - {x}_a).
\end{equation}
In addition, the Hessian of $\log Z({x}, {\lambda})$ with respect to ${\lambda}$ follows as
\begin{equation}
    {J}({x}, {\lambda})
    \equiv
    \frac{\partial^2 }{\partial{\lambda}^2} \log Z({x}, {\lambda})
    =
    \sum_{a \in I}
    w_a({x}, {\lambda}) ({x} - {x}_a)
    \otimes
    ({x} - {x}_a) - { r}({x}, {\lambda})
    \otimes
    {r}({x}, {\lambda}).
\end{equation}
Since ${ r}({x}, {\lambda}^*({x})) = {0}$,
\begin{equation}\label{eq:: def J*}
    {J^*}({x})
    \equiv
    {J}({x}, {{{\lambda}} ^*}({x}))
    =
    \sum_{a \in I} w_a^*({x}) ({x} - {x}_a)\otimes ({x} - {x}_a).
\end{equation}
It can be shown that ${J^*}({x})$ is positive definite.  In addition, the optimal shape functions $w^*_a\colon\conv(P) \to \mathbb{R}$ are $C^{\infty}$ and have gradient
\begin{equation}\label{eq:: grad p}
    \nabla w^*_a({x})
    =
    - w^*_a({x}) ({J^*}({x}))^{-1}({x} - {x}_a).
\end{equation}
We refer the reader to~\cite{arroyo_ortiz_06} for the proofs of the preceding results and identities.

The following lemma shows that, for a point set $P$ that is an {\it h}-covering of its closed convex hull $\overline{\conv}(P)$, and for every point ${x} \in \conv(P)$, for any vector ${\lambda}\neq {0}$ there exists at least one node point ${x}_{\lambda}$ in $P$ that is close to ${x}$, and such that ${x} - {x}_{\lambda}$ is closely aligned with ${\lambda}$.

\begin{lemma}\label{lemma:: x-lambda}
Let $P$ be a finite point set that is an {\it h}-covering of its convex hull $\overline{\conv}(P)$ for some $h > 0$. Let ${x}$ be a point in $\conv(P)$. Let $\varepsilon >0$ be such that ${\bar B}({x}, \varepsilon h) \subset \overline{\conv}(P)$. Let ${\lambda}\neq 0 \in \mathbb{R}^d$. Then, there exists a node point ${x}_{\lambda} \in P$ such that
\begin{itemize}
    \item[i)] $\varepsilon h \leq |{x} - {x}_{\lambda}| \leq (\varepsilon + 1)h$,
    \item[ii)] $|{\lambda}| \varepsilon h \leq \langle{\lambda}, {x} - {x}_{\lambda}\rangle$.
\end{itemize}
\end{lemma}
\begin{proof}
Let $\tilde{{x}}$ be the point $\tilde{{x}} = {x} - \frac{\varepsilon h}{|{\lambda}|} {\lambda}$. Since the distance between ${x}$ and $\tilde{{x}}$ is $ \varepsilon h$, $\tilde{{x}} \in {\bar B}({x}, \varepsilon h) \subset \overline{\conv}(P)$. In particular, since the point set $P$ is an {\it h}-covering of $\overline{\conv}(P)$, there exists a $d$-simplex $T_{\tilde{{x}}}$ of size at most $h$, with vertices in $P$ that contains $\tilde{{x}}$. Let $H^+ \subset \mathbb{R}^d$ be the halfspace $\{{z} \in \mathbb{R}^d\colon \langle {\lambda}, \tilde{{x}} - {z}\rangle \geq 0\}$. The point $\tilde{{x}}$ belongs to $H^+$. Moreover, since the $d$-simplex $T_{\tilde{{x}}}$ contains $\tilde{{x}}$, it follows that at least one extreme point of $T_{\tilde{{x}}}$ also belongs to $H^+$. Let ${x}_{\lambda}$ be that extreme point. Note that ${x}_{\lambda}$ is also a node point of $P$. We have the estimate
\begin{equation}
    |{x} -{x}_{\lambda}|
    \leq
    |{x} - \tilde{{x}}|
    +
    |\tilde{{x}} - {x}_{\lambda}|
    \leq
    \varepsilon h + h
    =
    (\varepsilon + 1) h.
\end{equation}
In addition, we have
\begin{equation}\label{eq:: identity_}
    |{x} -{x}_{\lambda}|^2
    =
    |{x} - {\tilde{x}}|^2 + |{\tilde{x}}-{x}_{\lambda}|^2
    +
    2\langle {x} -{\tilde{x}}, {\tilde{x}} -{x}_{\lambda} \rangle.
\end{equation}
By the definition of ${\tilde{x}}$, and since ${x}_{\lambda}$ belongs to $H^+$, it follows that
\begin{equation}
    \langle {x} -{\tilde{x}}, {\tilde{x}} -{x}_{\lambda} \rangle
    =
    \frac{\varepsilon h}{|{\lambda}|}\langle {\lambda},
    {\tilde{x}} -{x}_{\lambda} \rangle \geq 0.
\end{equation}
From this inequality and eq.~\eqref{eq:: identity_} we obtain
\begin{equation}
    |{x} -{x}_{\lambda}|^2 \geq |{x} - {\tilde{x}}|^2
    =
    (\varepsilon h)^2,
\end{equation}
or $|{x} -{x}_{\lambda}| \geq \varepsilon h$. Finally, from the definition of $\tilde{x}$ we have
\begin{equation}
    \langle {\lambda},{x} - {\tilde{x}}\rangle
    =
    |{\lambda}|\varepsilon h
\end{equation}
and
\begin{equation}
    \langle {\lambda},{x} - {x}_{\lambda}\rangle
    =
    \langle {\lambda},{x} - {\tilde{x}}\rangle
    +
    \langle {\lambda},{\tilde{x}} - {x}_{\lambda}\rangle
    \geq
    |{\lambda}|\varepsilon h,
\end{equation}
where we have used that ${x}_{{\lambda}}$ belongs to $H^+$ and,
hence, $\langle {\lambda},{\tilde{x}} - {x}_{\lambda}\rangle \geq
0$.
\end{proof}

In view of (\ref{eq:: optimal probas}), in order to verify that the LME shape functions have polynomial decay we require a bound on the minimizer ${\lambda}^*({x}) \in \mathbb{R}^d$ of the partition function $Z({x}, { {{\lambda}}})$, eqs.~\eqref{def:: partition function} and~\eqref{eq:: optimal lambda}. To this end, we begin with the following lemma.

\begin{lemma}\label{lemma:: UB Z(0)}
Let $P$ be a point set that is an {\it h}-covering of $\Omega$ with {\it h}-density bounded by $\tau$, for some $h$, $\tau > 0$. Let $\beta = \frac{\gamma}{h^2}$ for some $\gamma > 0$. Then, there exists a constant $c_{Z}$ that depends on $\gamma$, $\tau$, and $d$, such that
\begin{equation}
    Z({x}, {0})\leq c_{Z}
\end{equation}
for every ${x} \in \conv(P)$.
\end{lemma}
\begin{proof}
For every nonnegative integer $t \geq 1$, let $U_t({x})$ be the subset of node points $U_t({x}) = \{{x_a} \in P\colon (t-1)h \leq |{x_a} - {x} |< th\}$. Then, by Proposition~\ref{prop:: UB number node points on rings} we have
\begin{equation}\label{eq:: lemma 1, eq 2}
\begin{split}
    Z({x},{0})
    &=
    \sum_{a \in I} e^{-\beta |{x} - {x}_a|^2}
    =
    \sum_{t = 1}^{\infty}\left( \sum_{{x}_a
    \in U_t({x})} e^{-\beta |{x} - {x}_a|^2}\right)
    \\ & \leq
    \sum_{t = 1}^{\infty}\left( \# U_t({x}) \,
    e^{-\beta (t-1)^2h^2}\right)
    \leq
    \sum_{t = 1}^{\infty} c t^{d-1} e^{-\gamma (t-1)^2}
    =
    c_Z.
\end{split}
\end{equation}
It is readily verified that the series of the right hand side is absolutely convergent. Moreover, because this series is defined in terms of $\gamma$, $\tau$, and $d$, its limit $c_Z$ also depends on $\gamma$, $\tau$, and $d$ only.
\end{proof}

By optimality, ${\lambda}^*({x})$ has the property that $Z({x}, {\lambda}^*({x})) \leq Z({x}, {0})$. This observation, combined with the upper bound on $Z({x}, {0})$ of Lemma~\ref{lemma:: UB Z(0)}, suffices to estimate $|{{\lambda}}^*({x})|$.

\begin{lemma}\label{lemma:: UB lambda}
Let $P$ be a point set that is an {\it h}-covering of $\Omega$ with {\it h}-density bounded by $\tau$, for some $h$, $\tau > 0$. Let $\beta = \frac{\gamma}{h^2}$ for some $\gamma > 0$ and $\varepsilon > 0$. Then, there exists a constant $c_{\lambda}>0$ that depends on $\gamma$, $\tau$, and $d$ only such that
\begin{equation}\label{lem:LambdaBound}
    |{\lambda}^*({x})|
    \leq
    \frac{c_{\lambda}}{\min\{\varepsilon, 1 \}h}
\end{equation}
for every point $x$ such that ${\bar B}({x}, \varepsilon h) \subset \overline{\conv}(P)$.
\end{lemma}
\begin{proof}
We note that, since $\log$ is an increasing function, ${\lambda}^*({x})$ also minimizes $Z({x}, {\lambda})$. Let $\varepsilon_2 = \min\{\varepsilon, 1 \}$. We proceed to find a constant $c_{\lambda}$ such that, if $|{\lambda}|\geq \frac{c_{\lambda}}{\varepsilon_2 h}$, then $Z({x},{\lambda}) > Z({x},{0})$. To this end, let ${\lambda} \neq {0}$ be a fixed vector. Since ${\bar B}({x}, \varepsilon_2 h) \subset \overline{\conv}(P)$ and since $P$ is an {\it h}-covering of $\overline{\conv}(P)$, by Lemma~\ref{lemma:: x-lambda} there exists a point ${x}_{\lambda} \in P$ such that $\varepsilon_2 h \leq |{x} - {x}_{\lambda}| \leq (\varepsilon_2 + 1)h$ and $\langle{\lambda}, {x} - {x}_{\lambda}\rangle \geq |{\lambda}|\varepsilon_2 h$. Using these inequalities and noting that $\varepsilon_2 \leq 1$, we further obtain
\begin{equation}\label{eq:: lemma 1, eq 1}
    Z({x},{\lambda})
    =
    \sum_{a \in I}
    e^{-\beta |{x} - {x}_a|^2 + \langle{\lambda}, {x} - {x}_a\rangle}
    \geq
    e^{
    -\beta |{x} - {x}_{\lambda}|^2
    + \langle{{\lambda}}, {x}
    - {x}_{\lambda}\rangle
    }
    \geq
    e^{-\beta (1+ \varepsilon_2 )^2 h^2 + \varepsilon_2 h |{\lambda}|} \geq
    e^{-4 \gamma + \varepsilon_2 h |{\lambda}|} .
\end{equation}
By Lemma~\ref{lemma:: UB Z(0)}, there exists a constant $c_Z$ that depends on $\gamma$, $\tau$, and $d$, such that $Z({x},{0}) \leq c_Z$. Combining this bound with eq.~\eqref{eq:: lemma 1, eq 1}, it follows that a sufficient condition for ${\lambda}$ not to be optimal is that $e^{-4 \gamma + \varepsilon_2 h |{\lambda}|} > c_Z$ or, equivalently,
\begin{align}
    |{\lambda}|
    &>
    \frac{\ln c_{Z} + 4\gamma}{\min\{\varepsilon, 1\} h}
    =
    \frac{c_{\lambda}}{\min\{\varepsilon, 1\} h}.
\end{align}
Therefore, (\ref{lem:LambdaBound}) is a necessary condition for ${\lambda}^*({x})$ to be optimal.
\end{proof}

We note that, for fixed $\varepsilon >0$ and for points $x$ at distance $\varepsilon$ or greater to the boundary of $\overline{\conv}(P)$, the upper bound (\ref{lem:LambdaBound}) is $\text{O}(h^{-1})$. By contrast, for points ${x} \in \overline{\conv}(P)$ arbitrarily close of the boundary of $\overline{\conv}(P)$, the right hand side of (\ref{lem:LambdaBound}) diverges. The following example shows that $|{\lambda}^*({x})|$ may indeed diverge near the boundary.

\begin{example}\label{example:: unstable lambda}
{\rm Let $\Omega = [a, b] \subset \mathbb{R}$, $h = b-a$ and let $P = \{a, b\}$ be a point set of $\Omega$. Let $\beta = \frac{\gamma}{h^2}$ for a some $\gamma >0$. The optimality condition for $\lambda^*(x)$ is
\begin{equation}
    \frac{\partial Z(x, \lambda)}{\partial \lambda}
    =
    e^{-\frac{\gamma}{h^2}(x - a)^2
    +
    \lambda^*(x)(x-a)}(x-a)
    +
    e^{-\frac{\gamma}{h^2}(x - a-h)^2
    +
    \lambda^*(x)(x-a-h)}(x-a-h)=0.
\end{equation}
For this condition we find
\begin{equation}\label{eq:: optimal lambda example 1}
    \lambda^*(x)
    =
    \frac{\log(a + h - x) - \log(x - a)}{h}
    +
    \frac{\gamma}{h^2}(2x - 2a - h).
\end{equation}
For a fixed $0 < \varepsilon <1$, and for points $x \in (a+\varepsilon h, a+h - \varepsilon h)$, we indeed have $|\lambda^*(x)| = \text{O}(h^{-1})$. However, $\lim_{x \to a^+}\lambda^*(x) = \infty$, and $\lim_{x \to b^-}\lambda^*(x) = - \infty$. We note that the LME shape functions for this case reduce to
\begin{subequations}
\begin{align}
    w_a^*(x) &= \frac{a + h -x}{h}, \label{eq:: w_a example}\\
    w_b^*(x) &= \frac{x-a}{h}. \label{eq:: w_b example}
\end{align}
\end{subequations}
In particular, the shape functions and their derivatives are bounded in $\Omega$ even though the value of $|\lambda^*(x)|$ is unbounded at the boundary. From a computational perspective, this example suggests that computing the shape functions and their derivatives using Equations~\eqref{eq:: optimal lambda} and~\eqref{eq:: optimal probas} may be unstable near the boundary, even if the shape functions and their derivatives are themselves well-behaved. In Section~\ref{sec:: boundary-density} we will examine the behavior of the shape functions near the boundary more thoroughly.}
\hfill$\square$
\end{example}

The following lemma supplies the requisite estimate of the partition function $Z$.

\begin{lemma}\label{lemma:: LB and UB of Z}
Under the assumptions of Lemma~\ref{lemma:: UB lambda}, there exist constants $m_{Z}$, $M_Z >0$ that depend on $\gamma$, $\tau$, $\varepsilon$, and $d$ only and such that
\begin{equation}\label{eq:: LB and UB of Z}
    m_Z \leq Z({x}, {\lambda}^*({x}))
    \leq
    M_Z
\end{equation}
for every point $x$ such that ${\bar B}({x}, \varepsilon h) \subset \overline{\conv}(P)$.
\end{lemma}
\begin{proof}
By optimality, $Z({x}, {\lambda}^*({x})) \leq Z({x}, {0})$ and, by Lemma~\ref{lemma:: UB Z(0)}, $Z({x}, {\lambda}^*({x})) \leq c_Z = M_Z$ for every ${x} \in \conv(P)$. Since $P$ is an {\it h}-covering of $\overline{\conv}(P)$, there exists a point ${x}_0 \in P$ at distance to ${x}$ less or equal to $h$. In addition, by Lemma~\ref{lemma:: UB lambda}, there exists a constant $c_{\lambda}$ such that $| {\lambda}^*({x})| \leq \frac{c_{\lambda}}{\varepsilon_2 h}$, where $\varepsilon_2 = \min\{\varepsilon, 1\}$. We thus have
\begin{equation}\label{eq:: LB Z}
    Z({x}, {\lambda}^*({x}))
    \geq
    e^{
    -\beta|{x} - {x}_0|^2
    +
    \langle{\lambda}^*({x}), {x} - {x}_0\rangle
    }
    \geq
    e^{-\gamma-|{\lambda}^*({x})| \, |{x} - {x}_0|}
    \geq
    e^{-\gamma- \frac{c_{\lambda}}{\varepsilon_2}} = m_Z
    > 0,
\end{equation}
as advertised.
\end{proof}

Recall that ${J^*}({x}) \in \mathbb{R}^{d\times d}$ is the Hessian of $\log Z(x, \lambda^*(x))$ with respect to $\lambda$, eq.~\eqref{eq:: def J*}. We proceed to estimate $\|{J^*}({x})^{-1}\|$.

\begin{lemma}\label{lemma:: UB on inv J}
Let $P$ be a point set that is an {\it h}-covering of $\Omega$ with {\it h}-density bounded by $\tau$, for some $h$, $\tau > 0$. Let $\beta = \frac{\gamma}{h^2}$ for some $\gamma > 0$. Let $\varepsilon > 0$. Let ${x}$ be such that ${\bar B}({x}, \varepsilon h) \subset \overline{\conv}(P)$. Then, there exists a constant $c_{J^{-1}} > 0$ that depends on $\tau$, $\gamma$, $\varepsilon$, and $d$ such that
\begin{equation}\label{eq:JinvBound}
    \|{J^*}({x})^{-1}\|
    \equiv
    \sup_{{ y}\neq 0}
    \frac{|{J^*}({x})^{-1}({ y})|}{|{ y}|}
    \leq
    c_{J^{-1}} h^{-2}.
\end{equation}
\end{lemma}
\begin{proof}
Let ${ u} \neq {0}$ be a fixed vector. Then, from eq.~(\ref{eq:: def J*}) we have
\begin{equation}\label{eq:: uJu rewritten}
    { u}^T {J^*}({x}) { u}
    =
    \frac{\sum_{a \in I}
    e^{
    -\beta |{x} - {x}_a |^2
    +
    \langle{\lambda}^*({x}), {x} - {x}_a\rangle}
    \langle { u}, {x} - {x}_a\rangle^2} {Z({x}, {\lambda}^*({x}))}.
\end{equation}
Next, we analyze the numerator and denominator of the right-hand side in turn. Let $\varepsilon_2 = \min\{\varepsilon, 1\}$. By Lemma~\ref{lemma:: x-lambda}, there exists a point ${x}_{ u} \in P$ such that $\varepsilon_2 h \leq |{x} - {x}_{ u}| \leq (\varepsilon_2 + 1) h$ and $\langle{ u}, {x} - {x}_{ u}\rangle \geq |{ u}| \varepsilon_2h$. Since ${\bar B}({x}, \varepsilon h) \subset \overline{\conv}(P)$, by Lemma~\ref{lemma:: UB lambda} there exists a constant $c_{\lambda}$ such that $|{\lambda}^*({x})|\leq \frac{c_{\lambda}}{\varepsilon_2h}$, and we have
\begin{equation}
    |\langle{\lambda}^*({x}), {x} - {x}_{ u}\rangle|
    \leq
    |{\lambda}^*({x})| \, |{x} - {x}_{ u}|
    \leq
    \frac{c_{\lambda}}{\varepsilon_2 h}(\varepsilon_2 + 1)h
    =
    \frac{c_{\lambda}}{\varepsilon_2}(\varepsilon_2 + 1).
\end{equation}
Hence,
\begin{equation}\label{eq:: LB of numerator of uJu}
\begin{split}
    &
    \sum_{a \in I}
    e^{
    -\beta |{x} - {x}_a |^2
    +
    \langle{\lambda}^*({x}), {x} - {x}_a\rangle
    }
    \langle { u}, {x} - {x}_a\rangle ^2
    \\ & \geq
    e^{
    -\beta |{x} - {x}_{ u} |^2
    +
    \langle{\lambda}^*({x}), {x} - {x}_{ u}\rangle
    }
    \langle{ u}, {x} - {x}_{ u}\rangle^2
    \geq
    e^{
    -(\varepsilon_2+ 1)^2\gamma
    -
    \frac{c_{\lambda}}{\varepsilon_2}(\varepsilon_2+ 1)}
    |{ u}|^2 \varepsilon_2^2 h^2.
\end{split}
\end{equation}
where we write $\beta = \frac{\gamma}{h^2}$. Combining the bound supplied by Lemma~\ref{lemma:: LB and UB of Z} with eq.~\eqref{eq:: LB of numerator of uJu}, we get
\begin{equation}\label{eq::uJu}
    |{ u}^T {J^*}({x}) { u}|
    \geq
    e^{
    -(\varepsilon_2+ 1)^2\gamma
    -
    \frac{c_{\lambda}}{\varepsilon_2}
    (\varepsilon_2+1)}
    \frac{\varepsilon_2^2}{M_Z}
    |{ u}|^2 h^2 = c_J |{ u}|^2 h^2,
\end{equation}
where $c_J = e^{-(\varepsilon_2+ 1)^2\gamma - \frac{c_{\lambda}}{\varepsilon_2}(\varepsilon_2+ 1)}\frac{\varepsilon_2^2}{M_Z} > 0$ depends on $\gamma$, $\tau$, $\varepsilon$, and $d$ only. Let $\lambda_{\rm min}({x})$ be the smallest eigenvalue of ${J^*}({x})$. Since ${J^*}({x})$ is positive-definite~\cite{arroyo_ortiz_06}, it follows that $\lambda_{\rm min}({x}) > 0$. Inequality (\ref{eq::uJu}) then implies that $\lambda_{\rm min}({x}) \geq c_J h^2$. Since $\|{J^*}({x})^{-1}\| = 1/\lambda_{\rm min}({x})$, the estimate (\ref{eq:JinvBound}) follows immediately with $c_{J^{-1}} = 1/c_J$.
\end{proof}

We are finally in a position to estimate the derivatives of the LME shape functions.

\begin{proposition}\label{prop:: bound on grad}
Let $P$ be a point set that is an {\it h}-covering of $\Omega$ with {\it h}-density bounded by $\tau$, for some $h$, $\tau > 0$. Let $\beta = \frac{\gamma}{h^2}$, for some $\gamma > 0$ and $\varepsilon > 0$. Let $W = \{w^*_a\colon a \in I\}$ be the optimal shape functions of the LME approximation scheme with node set $P$ and parameter $\beta$.  Then,
\begin{equation}\label{eq:DwBound}
    |\nabla w^*_a({x})|
    \leq
    c_{J^{-1}} \, w^*_a({x}) |{x} - {x}_a| \, h^{-2},
\end{equation}
for every point $x$ such that ${\bar B}({x}, \varepsilon h) \subset \overline{\conv}(P)$ and every point ${x}_a \in P$.
\begin{proof}
The estimate (\ref{eq:DwBound}) follows immediately from Lemma~\ref{lemma:: UB on inv J} and eq.~\eqref{eq:: grad p}.
\end{proof}
\end{proposition}

Next we show that the LME approximation scheme has polynomial decay of order $(1, s)$ for every $s \geq 1$.

\begin{proposition}\label{prop:: LME-shape functions have poly-decay}
Let $\{I,W,P\}$ be an LME approximation scheme. Suppose that $P$ is an $h$-covering of $\Omega$, $P$ has $h$-density bounded by $\tau$, $\beta = \gamma / h^2$ for some $\gamma > 0$. Let $\varepsilon > 0$, and $s \geq 1$. Then, there exists a constant $c>0$ (depending on $d$, $\gamma$, $\tau$, $\varepsilon$, and $s$) such that the approximation scheme has polynomial decay of order $(1, s)$ for $c$ and $h$ in $\Omega_{\varepsilon h} = \{x\in \mathbb{R}^d \text{ s.~t. } {\bar B}({x}, \varepsilon h) \subset \overline{\conv}(P)\}$.
\end{proposition}
\begin{proof}
We recall that the LME shape functions are $C^{\infty}$ on $\conv (P)$ (\cite{arroyo_ortiz_06}). Next, we show that there exists a constant $c > 0$ that depends on $\gamma$, $\tau$, $d$, $\varepsilon$,  and $s$, such that, for any $k$,
\begin{equation}\label{eq:: poly decay 1}
    \sup_{|\alpha|\leq 1}
    \sup_{x \in \Omega_{\varepsilon h}}
    \sup_{a \in I}
    \left(1 + \left|\frac{x - x_a}{h}\right|^2\right)^{s}
    h^{|\alpha|} \left|D^{\alpha}w_a^*({x}) \right|
    \leq
    c.
\end{equation}
From Lemmas~\ref{lemma:: UB lambda} and~\ref{lemma:: LB and UB of Z} we have
\begin{equation}\label{eq:: exp decay}
\begin{split}
    0
    &\leq w^*_a({x})
    =
    \frac{e^{
    -\beta|{x} - {x}_a|^2
    +
    \langle{\lambda}^*({x}), {x} - {x}_a \rangle}}{Z({x},
    {\lambda}^*({x}))}
    \\ & \leq
    \frac{e^{
    -\gamma|({x} - {x}_a)/h|^2
    + |{\lambda}^*({x})| |{x} - {x}_a|}}{m_Z}
    \leq
    \frac{e^{
    -\gamma|({x} - {x}_a)/h|^2
    +
    \tilde{c}_{\lambda}|({x} - {x}_a)/h|}}{m_Z},
\end{split}
\end{equation}
with $\tilde{c}_{\lambda} = c_{\lambda}/\min\{\varepsilon, 1 \}$. In addition,
\begin{align}\label{eq:: bound for poly decay order 0}
    {}&
    \sup_{x \in \Omega_{\varepsilon h}}
    \sup_{a \in I}
    \left(1 + \left|\frac{x - x_a}{h}\right|^2
    \right)^{s} w_a^*({x})
    \nonumber \\ & \leq
    \sup_{x \in \Omega_{\varepsilon h}}
    \sup_{a \in I}
    \left(1 + \left|\frac{x - x_a}{h}\right|^2\right)^{s}
    \frac{e^{-\gamma|({x} - {x}_a)/h|^2 +
    \tilde{c}_{\lambda}
    |({x} - {x}_a)/h|}}{m_Z}
    \nonumber \\ & \leq
    c^{\prime}
    :=
    \sup_{t \geq 0} \left(1 + t^2\right)^{s}
    \frac{e^{-\gamma t^2 + \tilde{c}_{\lambda} t}}{m_Z}
    <
    \infty,
\end{align}
since $\frac{e^{-\gamma t^2 + \tilde{c}_{\lambda} t}}{m_Z}$ is a rapidly decreasing function of $t$. We note that $c^{\prime}$ is defined in terms of the constants $d$, $\gamma$, $\tau$, $\varepsilon$, and $s$. Thus, by Proposition~\ref{prop:: bound on grad} we have
\begin{align}\label{eq:: bound for poly decay order 1}
    {}&\sup_{x \in \Omega_{\varepsilon h}}
    \sup_{a \in I}
    \left(1 + \left|\frac{x - x_a}{h}\right|^2 \right)^{s}h
    \left|\nabla w_a^*({x}) \right|
    \nonumber \\ & \leq
    \sup_{x \in \Omega_{\varepsilon h}}
    \sup_{a \in I}
    \left(1 + \left|\frac{x - x_a}{h}\right|^2 \right)^{s}h
    c_{J^{-1}} w^*_a({x})|{x} - {x}_a| h^{-2}
    \nonumber\\ & \leq
    \sup_{x \in \Omega_{\varepsilon h}}
    \sup_{a \in I}
    \left(1 + \left|\frac{x - x_a}{h}\right|^2 \right)^{s}
    c_{J^{-1}} w^*_a({x})
    <
    c_{J^{-1}} c^{\prime},
\end{align}
as advertised.
\end{proof}

The next Theorem bounds uniformly the error of the approximate function $u_k$ and its derivatives to a smooth function $u$ and its derivatives. The result is based on Theorem~\ref{thm:: uniform error bound} that holds for a general approximation scheme.

\begin{theorem}\label{thm:: error bound of LME approximates}
Under the assumptions of Proposition~\ref{prop:: LME-shape functions have poly-decay}, let $u \in C^2(\bar{\Omega})$. Then, there exists a constant $C >0$, that depends on $\gamma$, $\tau$, $\varepsilon$, and $d$ only, such that
\begin{equation}
    \left|D^{\alpha} u_I({x}) - D^{\alpha} u({x}) \right|
    \leq C \|D^2 u\|_{\infty} h^{2 - |\alpha|},
\end{equation}
for ${x} \in \Omega_{\varepsilon h}$, $|\alpha| \leq 1$.
\end{theorem}
\begin{proof}
The theorem follows from Proposition~\ref{prop:: LME-shape functions have poly-decay} and Theorem~\ref{thm:: uniform error bound}.
\end{proof}

Finally, we are in a position to show that LME approximation spaces on a domain $\Omega^{\prime}$ are dense in $W^{1,p}(\Omega)$ for subdomains $\Omega \subset \Omega^{\prime}$ which are compactly contained in $\Omega^{\prime}$. This result is derived from the polynomial decay of LME schemes, and the density of approximation schemes of Corollary~\ref{coro:: density in Sobolev space}.

\begin{corollary}\label{cor:: H1 density}
Let $\Omega$ be a domain satisfying the segment condition, and let $\Omega^{\prime}$ be an auxiliary domain such that $\overline{\Omega} \subset \Omega^{\prime}$. Let $\{I_k,W_k,P_k\}$ be a sequence of LME approximation schemes in $\Omega^{\prime}$. Suppose that the assumptions of Proposition~\ref{prop:: LME-shape functions have poly-decay} hold for $\{I_k,W_k,P_k\}$ in $\Omega^{\prime}$ uniformly for $h_k\to 0$. Let $ 1\leq p < \infty$. Then, for every $u \in W^{1,p}(\Omega)$ there exists a sequence $u_k \in X_k$ such that $u_k{}_{|\Omega} \to u$.
\end{corollary}
\begin{proof}
As $\overline{\Omega} \subset \Omega^{\prime}$, there exists $r>0$ such that, $\cup_{x \in \Omega} B(x, r) = \{x \in \R^d\colon {\rm dist}(x, \Omega) < r\} \subset \Omega^{\prime}$. The sequence of approximation schemes $\{I_k,W_k,P_k\}$, when restricted to $\Omega$, has uniform polynomial decay $(1,s)$ for any fixed $s$. Then, the theorem follows from Corollary~\ref{coro:: density in Sobolev space}.
\end{proof}
Corollary~\ref{cor:: H1 density} guarantees the density of the LME
approximates on $W^{1,p}(\Omega)$, provided that the sequence of LME
approximation schemes $\{I_k,W_k,P_k\}$ is defined on a bigger
domain $\Omega^{\prime}$. We note that, in this case, the LME scheme
does not obey the weak Kronecker-delta property at the boundary of
$\Omega$, making it less straightforward to enforce boundary
conditions on $\Omega$. However, imposing boundary conditions can be
done in this case by using standard Lagrangian multipliers, see
e.~g.~\cite{meshfree_review, cottrell_06}.
\end{section}

\begin{section}{Application to the Local Maximum-Entropy (LME) Approximation Scheme: Estimates up to the Boundary} \label{sec:: boundary-density}

In Section~\ref{sec:: convergence analysis} we have seen that, for a
sequence $\{I_k, W_k, P_k\}$ of LME approximation schemes, we have
density of the approximation space $X_k$ in $W^{1, p}_{\rm
loc}(\Omega)$. In order to treat boundary value problems, however,
we need density results up to the boundary of $\Omega$. A way to
guarantee the density in $W^{1,p}(\Omega)$ is to work with a
sequence $\{I_k, W_k, P_k\}$ defined on a (strictly) bigger domain
$\Omega^{\prime}$, as discussed in Corollary~\ref{cor:: H1 density}.
In this section, we analyze the density of the approximation space
$X_k$ when the domain of the LME scheme is $\Omega$.

While we will see that density can be extended to $W^{1,
p}_0(\Omega)$ in general, a major technical difficulty with
estimates up to the boundary comes from the fact that $\lambda^*(x)$
blows up as $x$ approaches $\partial \Omega$. This blowup is indeed
a manifestation of the weak Kronecker-delta property at the
boundary, as $\lambda^*$ will blow up in such a way that in the
limit no weight is given to nodal data in the interior of $\Omega$.
For general $\Omega$, this behavior can become very complicated and
lead to blow up of the gradients of the optimal shape functions
$\nabla w_a$, with the result that the general convergence scheme of
Section~\ref{sec:: general convergence analysis} is no longer
applicable. Therefore, for simplicity we restrict attention to the
class of polyhedral domains. Under generic assumptions, we shall
obtain sufficiently strong estimates on $\nabla w^*$ near flat
pieces of the boundary $\partial \Omega$ permitting to show that,
away from a small singular part of the boundary, Sobolev functions
can be approximated by linear combinations of shape functions in the
limit of $h \to 0$. The singular boundary is of finite $2$-capacity.
With the help of a capacity argument we can then establish
approximation results with truncated LME functions in $H^1(\Omega)$
for spatial dimension $d > 2$.

More precisely, in this section we will assume that $\Omega$ is a
convex polytope in $\R^d$, $P$ is an $h$-covering for $\Omega$ with
$\operatorname{conv} P = \Omega$ such that there exists a constant
$\eta > 0$ such that $\{x \in P \colon 0 < {\rm dist}(x, \partial
\Omega) < \eta h \} = \emptyset$. Note that then $P \cap \partial
\Omega$ is an $h$-covering for $\partial \Omega$.

Assume that $A = H \cap \partial \Omega$, $H$ some hyperplane, is a
flat $(d-1)$-dimensional subset of the boundary of $\Omega$. With
the aim to control $\nabla w^*(x)$ for $x$ in the vicinity of $A$,
our first task will be to exactly estimate the behavior of $J^*(x)$
in this regime. First note that with a proper choice of the
coordinate system we may assume that $H = \{x_1 = 0\} = \{0\} \times
\R^{d - 1}$ with $\Omega \cap \{x_1 \ge 0\} = \Omega$. Accordingly,
we write
\begin{equation}
    \lambda^*(x)
    =
    (\lambda^*_1(x), \lambda'(x)) \in \R \times \R^{d - 1}
\end{equation}
and, for $x = (x_1, x')$,
\begin{align}\label{eq:Z}
    Z
    =
    \sum_{a \in I} e^{-\beta|x - x_a|^2 + \langle \lambda', x' - x_a' \rangle + (x - x_a)_1 \lambda^*_1}.
\end{align}

We fix $\delta > 0$ and consider points $x \in \Omega$ with $x =
(\rho, x') \in \R \times \R^{d - 1}$ for $\rho$ small such that
$B_{\delta h}(0,x') \cap H = B_{\delta h}(0,x') \cap A$. In the
following lemmas we will also set $h = 1$ for arbitrarily large
$\Omega$ and recover the general case in
Proposition~\ref{prop:blow-up-est} by rescaling afterwards. Generic
positive constants, denoted $c$, $c'$, $c''$ or $C$, $C'$, will be
independent of $\rho$ and the size of $\Omega$.

\begin{lemma}\label{lemma:lambda-t-bound}
There is a constant $C > 0$ such that
\begin{equation}
    |\lambda'(x)|
    \le
    C \quad
    \text{and} \quad
    \lambda^*_1(x)
    \ge
    -C.
\end{equation}
\end{lemma}

\begin{proof}
This result follows along the same lines as the proof of
Lemma~\ref{lemma:: UB lambda} for the boundedness of $\lambda^*$ in
the interior of $\Omega$.
\end{proof}

In order to investigate $Z$ we split the sum as
\begin{align}\label{eq:Z-split}
    Z
    =
    \sum_{a \in I_A} e^{-\beta|x - x_a|^2 + \langle \lambda', x' - x_a' \rangle + \rho \lambda^*_1}
    + \sum_{a \in I \setminus I_A} e^{-\beta|x - x_a|^2 + \langle \lambda', x' - x_a' \rangle - (x_a - x)_1 \lambda^*_1},
\end{align}
where $I_A$ collects those indices $a$ for which $x_a \in A$.

\begin{lemma}
As $\rho$ tends to $0$, $\lambda^*_1 \to \infty$ such that $\rho \lambda^*_1 \to 0$.
\end{lemma}

\begin{proof}
Writing $Z$ as in \eqref{eq:Z-split} and noting that $(x_a - x)_1 > \eta$ for $a \notin I_A$ and $\rho > 0$, we see that
\begin{equation}
    Z
    \ge
    \sum_{a \in I_A} e^{-\beta|x - x_a|^2 + \langle \lambda', x' - x_a' \rangle}
\end{equation}
and that this lower bound is in fact achieved only if $\rho \lambda^*_1 \to 0$ and $\lambda^*_1 \to \infty$.
\end{proof}
In particular, we see that $Z$ still remains bounded from above and from below by positive constants.

In order to estimate $J^*$, we first observe that the optimality condition $\frac{\partial Z}{\partial \lambda_1} = 0$ implies
\begin{align}\label{eq:opt-cond}
    \rho \sum_{a \in I_A} e^{-\beta|x - x_a|^2
    \langle \lambda', x' - x_a' \rangle + \rho \lambda^*_1}
    =
    \sum_{a \in I \setminus I_A}
    e^{-\beta|x - x_a|^2 + \langle \lambda', x' - x_a' \rangle - (x_a - x)_1 \lambda^*_1}
    (x_a - x)_1.
\end{align}

\begin{lemma}\label{lemma:J11-lower-bound}
There is a constant $c > 0$ such that the first entry $J^*_{11}$ in $J^*(x)$ satisfies
\begin{equation}
    J^*_{11}
    \ge
    c \rho.
\end{equation}
\end{lemma}

\begin{proof}
Since $(x_a - x)_1 \ge \eta$ for $a \notin I_A$ we find by eq.~\eqref{eq:opt-cond}
and Lemma~\ref{lemma:lambda-t-bound}
\begin{equation}
\begin{split}
    J^*_{11}
    &= \sum_{a \in I} w^*_a(x) (x - x_a)_1^2 \\
    &\ge Z^{-1} \sum_{a \in I \setminus I_A}
       e^{-\beta|x - x_a|^2 + \langle \lambda', x' - x_a' \rangle - (x_a - x)_1 \lambda^*_1}
       (x_a - x)_1^2 \\
    &\ge \eta  Z^{-1} \sum_{a \in I \setminus I_A}
       e^{-\beta|x - x_a|^2 + \langle \lambda', x' - x_a' \rangle - (x_a - x)_1 \lambda^*_1}
       (x_a - x)_1 \\
    &= \eta \rho Z^{-1} \sum_{a \in I_A}
       e^{-\beta|x - x_a|^2 + \langle \lambda', x' - x_a' \rangle + \rho \lambda^*_1} \\
    &\ge c\rho
\end{split}
\end{equation}
for some constant $c > 0$.
\end{proof}

We now derive an upper bound for the entries of the first row and column of $J^*$.

\begin{lemma}\label{lemma:J1j-upper-bound}
For any $0 < \mu < 1$ there exists a constant $C > 0$ such that
\begin{equation}
    |J^*_{1j}|
    =
    |J^*_{j1}| \le C \rho^{\mu}, \qquad j = 1, \ldots, d.
\end{equation}
\end{lemma}

\begin{proof}
For $j = 1, \ldots, d$ we have $J^*_{1j} = J^*_{j1} = \sum_{a \in I} w^*_a(x) (x - x_a)_1 (x - x_a)_j$. First summing over $a \in I_A$ gives the obvious bound \begin{align}\label{eq:first-summand-est}
\begin{split}
    \left| \sum_{a \in I_A} w^*_a(x) (x - x_a)_1 (x - x_a)_j \right|
    &\le Z^{-1} \sum_{a \in I_A} e^{-\beta|x - x_a|^2 + \langle \lambda', x' - x_a' \rangle + \rho \lambda^*_1} \rho |x - x_a| \\
    &\le C \rho.
\end{split}
\end{align}

In order to estimate the remaining sum, we let $p = \frac{1}{\mu}$ and choose $1 < q < \infty$ with $\frac{1}{p} + \frac{1}{q} = 1$, so that
\begin{align}\label{eq:Hoelder-Absch}
    \left| \sum_{a \in I \setminus I_A} w^*_a(x) (x - x_a)_1 (x - x_a)_j \right|
    &\leq \left( \sum_{a \in I \setminus I_A} w^*_a(x) (x_a - x)_1^p \right)^{\frac{1}{p}}
     \left( \sum_{a \in I \setminus I_A} w^*_a(x) |x - x_a|_j^q \right)^{\frac{1}{q}}
\end{align}
by H{\"o}lder's inequality. Here the second factor in \eqref{eq:Hoelder-Absch} is bounded by
\begin{align}
    \left( Z^{-1} \sum_{a \in I \setminus I_A}
    e^{-\beta|x - x_a|^2 + \langle \lambda', x' - x_a' \rangle - (x_a - x)_1 \lambda^*_1} |x - x_a|_j^q \right)^{\frac{1}{q}}
    &\le
    C,
\end{align}
see Lemma~\ref{lemma:lambda-t-bound}. To estimate the first factor we note that, since $P$ is a $1$-covering of $\Omega$, there exists $\bar{a} \in I \setminus I_A$ such that
\begin{equation*}
    |x_{\bar{a}} - x|
    \le
    C'
\end{equation*}
for a constant $C' > \eta$. For $\rho$ sufficiently small and thus $\lambda^*_1$ sufficiently large, we then have the estimate
\begin{equation}
\begin{split}
    \sum_{\substack{a \in I \setminus I_A \\ (x_a - x)_1 \ge C'}} w^*_a(x) (x_a - x)_1^p
    &\le
    Z^{-1} \sum_{\substack{a \in I \setminus I_A \\ (x_a - x)_1 \ge C'}}
    e^{-\beta|x - x_a|^2 + \langle \lambda', x' - x_a' \rangle - C' \lambda^*_1} (x_a - x)_1^p \\
    &\le C Z^{-1} e^{- C' \lambda^*_1} \\
    &\le C Z^{-1} e^{-\beta|x - x_{\bar{a}}|^2 + \langle \lambda', x' - x_{\bar{a}}' \rangle - (x_{\bar{a}} - x)_1 \lambda^*_1} \\
    &\le C Z^{-1} \eta^{-1} \sum_{a \in I \setminus I_A}
    e^{-\beta|x - x_a|^2 + \langle \lambda^*, x - x_a \rangle} (x_a - x)_1,
\end{split}
\end{equation}
as $(x_a - x)_1 \ge \eta$ for $a \in I \setminus I_A$. On the other hand, for $a$ with $(x_a - x)_1 \le C'$ we have the bound
\begin{equation}
    (x_a -x)_1^p
    \le
    C(x_a - x)_1.
\end{equation}
Combining the two last estimates, we see that the term in the first factor of \eqref{eq:Hoelder-Absch} satisfies
\begin{align}
    \sum_{a \in I \setminus I_A} w^*_a(x) (x - x_a)_1^p
    &\le C \sum_{a \in I \setminus I_A} w^*_a(x) (x - x_a)_1.
\end{align}
Since by \eqref{eq:opt-cond} this last expression is bounded by $C \rho$, we arrive at
\begin{align}
    \left| \sum_{a \in I \setminus I_A} w^*_a(x) (x - x_a)_1 (x - x_a)_j \right|
    \le C \rho^{\frac{1}{p}}
    = C \rho^{\mu}
\end{align}
by \eqref{eq:Hoelder-Absch}. Together with the bound~\eqref{eq:first-summand-est} for the first part of the sum we have shown that indeed
\begin{equation*}
    |J^*_{1i}|
    \le
    C \rho^{\mu}.
\end{equation*}
\end{proof}

For the remaining part $B = (J^*_{ij})_{2 \le i,j \le n}$ of the matrix $J^*$ we obtain the following lower matrix bound.

\begin{lemma}\label{lemma:B-lower-bound}
There is a constant $c > 0$ such that
\begin{equation}
    B
    \ge
    c \operatorname{Id}_{n-1} .
\end{equation}
\end{lemma}

\begin{proof}
As $P \cap H$ is a $1$-covering for $A$, there is a set $J = \{ a_1,
\ldots, a_{d-1} \} \subset I_A$ of $d-1$ points such that $c' \le
|x' - x'_a| \le c''$ and $\det(x' - x'_{a_1}, \ldots, x' -
x'_{a_{d-1}}) \ge c'$ for suitable constants $c'$ and $c''$. Then
\begin{equation}
\begin{split}
    B
    &=
    \sum_{a \in I} w^*_a(x) (x' - x_a') \otimes (x' - x_a')\\
    &\ge
    Z^{-1} \sum_{a \in J} e^{-\beta|x - x_{\bar{a}}|^2 + \langle \lambda', x' - x_{\bar{a}}' \rangle + \rho \lambda^*_1}
       (x' - x_a') \otimes (x' - x_a')\\
    &\ge
    c \sum_{a \in J} (x' - x_a') \otimes (x' - x_a') \\
    &\ge
    c \operatorname{Id}_{d-1}
\end{split}
\end{equation}
since all the projections $(x' - x_a') \otimes (x' - x_a')$ are nonnegative.
\end{proof}

As a consequence of the above results, we obtain an estimate for the inverse matrix $(J^*)^{-1} = (\tilde{J}_{ij})$.

\begin{lemma}\label{lemma:J-inverse-bounds}
For any $0 < \mu < \frac{1}{2}$ there exists a constant $C$ such that
\begin{align}
    |\tilde{J}_{ij}|
    \begin{cases}
    \le C \rho^{-1} &\text{for } i = j = 1, \\
    \le C \rho^{- \mu} &\text{for } i = 1, j = 2, \ldots, d \text{ or } j = 1, i = 2, \ldots, d \quad \text{and}\\
    \le C &\text{for } i,j = 2, \ldots, d.
    \end{cases}
\end{align}
\end{lemma}

\begin{proof}
First note that, expanding with respect to the first row, for $\frac{1}{2} < \tilde{\mu} < 1$ we have
\begin{equation}
    \det J^*
    =
    J^*_{11} \det B + O(\rho^{2 \tilde{\mu}})
   = J^*_{11} \det B
   \ge c J^*_{11}
   \ge c \rho
\end{equation}
by Lemmas~\ref{lemma:J11-lower-bound}, \ref{lemma:J1j-upper-bound} and~\ref{lemma:B-lower-bound}. Furthermore, as $|J^*| \le C$, we have
\begin{align}
    |(\operatorname{cof} J^*)_{ij}|
    \begin{cases}
    = |\det B| \le C &\text{for } i = j = 1, \\
    \le C \rho^{\tilde{\mu}} &\text{for } i = 1, j \ge 2 \text{ or } j = 1, i \ge 2 \quad \text{and} \\
    \le C (J^*_{11} + \rho^{2 \tilde{\mu}}) &\text{for } i,j \ge 2.
    \end{cases}
\end{align}
for $C$ sufficiently large. Now, Cramer's rule
\begin{align}
  (J^*)^{-1}
  =
  (\det J^*)^{-1} (\operatorname{cof} J^*)^T
\end{align}
implies
\begin{align}
  |\tilde{J}_{ij}|
  \begin{cases}
    \le C \rho^{-1} &\text{for } i = j = 1, \\
    \le C \rho^{\tilde{\mu} - 1} &\text{for } i = 1, j \ge 2 \text{ or } j = 1, i \ge 2 \quad \text{and}\\
    \le C &\text{for } i,j \ge 2
  \end{cases}
\end{align}
and thus the assertion follows by choosing $\tilde{\mu}$ such that $\mu = 1 -\tilde{\mu}$.
\end{proof}

\begin{lemma}\label{lemma:nabla-w-control}
For any $s > 0$ and $0 < \mu < \frac{1}{2}$ there is a constant $C > 0$ such that
\begin{equation}
    \left(1 + |x - x_a|^2 \right)^s |\nabla w_a^*(x)|
    \le
    C (1 + \rho^{-\mu}|x - x_a|).
\end{equation}
\end{lemma}

\begin{proof}
If $a \in I_A$, then $x - x_a = (\rho, x' - x_a')$ and Lemma~\ref{lemma:J-inverse-bounds} shows
\begin{align}
  \left| (J^*)^{-1} (x - x_a) \right|
  \le C(1 + \rho^{-\mu}|x' - x_a'|)
  \le C(1 + \rho^{-\mu}|x - x_a|).
\end{align}
So, by \eqref{eq:: grad p},
\begin{align}
  |\nabla w^*_a(x)|
  \le C |w^*_a(x)| (1 + \rho^{-\mu}|x - x_a|).
\end{align}
Now, using that $(1 + |x - x_a|^2)^s |w^*_a(x)| \le C$ for any $a$, we see that the estimate holds true for $a \in I_A$.

On the other hand, if $a \notin I_A$, then Lemma~\ref{lemma:J-inverse-bounds} only gives
\begin{align}
  \left| (J^*)^{-1} (x - x_a) \right|
  \le C \rho^{-1}|x - x_a|,
\end{align}
whence
\begin{align}
  |\nabla w^*_a(x)|
  \le C |w^*_a(x)| \rho^{-1}|x - x_a|.
\end{align}
But since $(x_b - x)_1 \ge \eta$ for all $b \notin I_A$, we also get
\begin{equation}
\begin{split}
    &(1 + |x - x_a|^2)^s |w^*_a(x)|\\
    &=
    (1 + |x - x_a|^2)^s Z^{-1} e^{-\beta|x - x_a|^2
    + \langle \lambda', x' - x_a' \rangle - (x_a - x)_1 \lambda^*_1} \\
    &\le
    Z^{-1} \eta^{-1} \sum_{b \in I \setminus I_A}
       e^{-\beta|x - x_b|^2 + \langle \lambda', x' - x_b' \rangle
       - (x_b - x)_1 \lambda^*_1} (x_b - x)_1 (1 + |x - x_b|^2)^s.
\end{split}
\end{equation}
This term can now be estimated by $C\rho^{\tilde{\mu}} \rho^{-1}$
for $0 < \tilde{\mu} < 1$ precisely as the left hand side of
\eqref{eq:Hoelder-Absch} in Lemma~\ref{lemma:J1j-upper-bound}, which
leads to
\begin{align}
  (1 + |x - x_a|^2)^s |\nabla w^*_a(x)|
  \le C \rho^{-\mu}|x - x_a|
\end{align}
for $0 < \mu < 1$.
\end{proof}

Undoing the rescaling of $h$ we can now summarize the previous lemmas in the following proposition the boundary behavior of $\nabla w_a^*$ near flat parts of $\partial \Omega$.

\begin{proposition}\label{prop:blow-up-est}
Suppose $x = (\rho, x') \in \R \times \R^{d - 1}$ is such that
$B_{\delta h}(0,x') \cap H = B_{\delta h}(0,x') \cap A$ for a
boundary $(d-1)$-face $A = \partial \Omega \cap H$. Let $s > 0$ and
$0 < \mu < \frac{1}{2}$. There is a constant $C > 0$ such that
\begin{align}
    \left(1 + \left|\frac{x - x_a}{h} \right|^2 \right)^s h |\nabla w^*_{a,h}(x)|
    &\le
    C \left( 1 + h^{\mu} {\rm dist}^{-\mu}\left(x, \partial \Omega \right) \right).
\end{align}
\end{proposition}

\begin{proof}
If $P$ is an $h$-covering for $\Omega$, then $h^{-1} P$ is a $1$-covering for $h^{-1} \Omega$. Using subscripts $h$ to highlight the dependence on $h$, we have
\begin{equation}
\begin{split}
    Z_h(x)
    &=
    \sum_{x_b \in P} e^{-\frac{\gamma}{h^2} |x - x_b|^2 + \langle \lambda^*_{(h)}, x - x_b \rangle} \\
    &=
    \sum_{x_b \in P} e^{-\gamma |\frac{x - x_b}{h}|^2 + \langle h \lambda^*_{(h)}, \frac{x - x_b}{h} \rangle} \\
    &=
    \sum_{x_b \in h^{-1} P} e^{-\gamma |\frac{x}{h} - x_b|^2 + \langle h \lambda^*_{(h)}, \frac{x}{h} - x_b
    \rangle}.
\end{split}
\end{equation}
This expression is minimized at $h \lambda^*_{(h)}(x) =
\lambda^*_{(1)}(\frac{x}{h})$. For the shape functions $w^*_{a,h}$
we denote by $w_{a,1}$ the shape function corresponding to the node
$\frac{x_a}{h} \in h^{-1}P$ and obtain
\begin{equation}
\begin{split}
  w^*_{a,h}(x)
  &= \frac{e^{-\frac{\gamma}{h^2} |x - x_a|^2 + \langle \lambda^*_{(h)}, x - x_a \rangle}}
     {\sum_{x_b \in P} e^{-\frac{\gamma}{h^2} |x - x_b|^2 +
     \langle \lambda^*_{(h)}, x - x_b \rangle}} \\
  &= \frac{e^{-\gamma |\frac{x - x_a}{h}|^2 + \langle h \lambda^*_{(h)}, \frac{x - x_a}{h} \rangle}}
     {\sum_{x_b \in P} e^{-\gamma |\frac{x - x_b}{h}|^2 +
     \langle h \lambda^*_{(h)}, \frac{x - x_b}{h} \rangle}} \\
  &= \frac{e^{-\gamma |\frac{x}{h} - \frac{x_a}{h}|^2 +
     \langle \lambda^*_{(1)}(\frac{x}{h}), \frac{x}{h} - \frac{x_a}{h} \rangle}}
     {\sum_{x_b \in h^{-1} P} e^{-\gamma |\frac{x}{h} - x_b|^2 + \langle \lambda^*_{(1)}(\frac{x}{h}), \frac{x}{h} - x_b \rangle}} \\
  &= w^*_{a, 1}\left(\frac{x}{h}\right).
\end{split}
\end{equation}

Applying Lemma~\ref{lemma:nabla-w-control} to  $w^*_{a, 1}$ we therefore obtain
\begin{equation}
\begin{split}
  \left(1 + \left|\frac{x - x_a}{h} \right|^2 \right)^s |\nabla w^*_{a,h}(x)|
  &= \left(1 + \left|\frac{x - x_a}{h} \right|^2 \right)^s h^{-1} \left| \nabla w^*_{a, 1}\left(\frac{x}{h}\right) \right| \\
  &\le C h^{-1} \left( 1 + {\rm dist}^{-\mu}\left(\frac{x}{h}, A\right)\left|\frac{x}{h} - \frac{x_a}{h}\right| \right).
\end{split}
\end{equation}
From our previous calculations we know already that the left hand
side is bounded by a constant away from the boundary of $\Omega$.
Suppose $x$ is such that ${\rm dist}^{-\mu}(\frac{x}{h}, A) \ge 1$.
Since $s$ is arbitrary we can absorb the last factor on the right
hand side into the prefactor of the left hand side and finally
obtain
\begin{align}
    \left(1 + \left|\frac{x - x_a}{h} \right|^2 \right)^s h |\nabla w^*_{a,h}(x)|
    &\le
    C \left( 1 + h^{\mu} {\rm dist}^{-\mu}\left(x, \partial \Omega \right) \right).
\end{align}
\end{proof}

Now suppose that $x$ is a general point near a possibly lower
dimensional edge of $\partial \Omega$. More precisely, $x$ is close
to an $m$-face of $A$ of $\partial \Omega$, which is the
intersection of $d - m$ hyperplanes $H_1, \ldots, H_{d-m}$ with
linearly independent normals which constitute $\partial \Omega$ in
the vicinity of $x$.

\begin{lemma}\label{lemma:singular-boundary}
There exists $R > 0$ such that for all $x_a \in P$ with ${\rm dist}(x_a, \partial \Omega) \ge Rh$
\begin{align}
  \left(1 + \left|\frac{x - x_a}{h} \right|^2 \right)^s h |\nabla w^*_{a,h}(x)|
  \le
  1.
\end{align}
\end{lemma}

\begin{proof}
We first assume again that $h = 1$. Let $H$ be the hyperplane containing $x$ which is perpendicular to $\lambda^*$. Similarly as in Lemma~\ref{lemma:lambda-t-bound} we see that, as ${\rm dist}(x, A) \to 0$, $|\lambda^*|$ tends to infinity such that the projection of $\lambda^*$ onto $\bigcap H_i$ remains bounded and that there are constants $c$, $C > 0$ such that \begin{equation}
    \left\langle y - x, \frac{\lambda^*}{|\lambda^*|} \right\rangle
    =
    {\rm dist}(y - x, H)
    \ge
    c\, {\rm dist}(y, H_1 \cup \ldots \cup H_{d-m}) - C
\end{equation}
for every $y \in \Omega$.

Choose a set $J = \{ a_1, \ldots, a_d \} \subset I$ of $d$ points
such that $c' \le |x - x_a| \le c''$ and $\det(x - x_{a_1}, \ldots,
x - x_{a_d}) \ge c'$ for suitable constants $c'$ and $c''$. Then
\begin{equation}
    J^*
    \ge Z^{-1} \sum_{a \in J} e^{-\beta|x - x_a|^2 + \langle \lambda^*, x - x_a \rangle}
       (x - x_a) \otimes (x - x_a)
    \ge c\, e^{-C|\lambda^*|} \operatorname{Id}_{d-1}
\end{equation}
since all the projections $(x - x_a) \otimes (x - x_a)$ are nonnegative. It follows that
\begin{equation}
    (J^*)^{-1}
    \le
    C e^{C|\lambda^*|} \operatorname{Id}_{d-1}.
\end{equation}

Now if $x_a \in \Omega$ satisfies ${\rm dist}(x, \partial \Omega) \ge R$, then
\begin{equation}
    \langle x - x_a, \lambda^* \rangle \le - (c\, {\rm dist}(x_a,
    H_1 \cup \ldots \cup H_{d-m}) + C) |\lambda^*|
    \le
    (- c R + C) |\lambda^*|.
\end{equation}
So
\begin{equation}
    |w^*_a|
    \le
    Z^{-1} e^{- \beta|x - x_a|^2 + \langle \lambda^*, x - x_a \rangle}
    \le
    Z^{-1} e^{- \beta|x - x_a|^2 - (c R - C) |\lambda^*|}.
\end{equation}
It follows from \eqref{eq:: grad p} that
\begin{equation}
\begin{split}
    \left(1 + \left|x - x_a \right|^2 \right)^s |\nabla w^*_a(x)|
    &\le
    C \left(1 + \left|x - x_a \right|^2 \right)^s e^{- \beta|x - x_a|^2 - (c R + C) |\lambda^*|} e^{C|\lambda^*|} |x - x_a| \\
    &\le
    C e^{(2C - c R)|\lambda^*|}
    \le 1
    \end{split}
\end{equation}
for $R$ sufficiently large, which proves the Lemma for $h = 1$. The estimate for general $h$ now follows directly by rescaling as before.
\end{proof}

We are now in a position to prove our main density results up to the
boundary. Density in $W^{1, p}_0(\Omega)$ in fact only relies on our
previous interior estimates, see Section~\ref{sec:: convergence
analysis}, and Lemma~\ref{lemma:singular-boundary} and is true for
general, not necessarily polyhedral domains $\Omega$.

\begin{theorem}\label{theorem:W1p0-density}
Let $\Omega$ be a bounded polyhedron and $\{I_k, W_k, P_k\}$ be a
sequence of LME approximation schemes satisfying the assumptions of
Proposition~\ref{prop:: LME-shape functions have poly-decay}
uniformly for $h_k\to 0$. Then for any $u \in W^{1, p}_0(\Omega)$,
$1 \le p < \infty$, there exists a sequence $u_k \in X_k$ such that
$u_k \to u$.
\end{theorem}

\begin{proof}
It suffices to consider $u \in C^{\infty}_c(\Omega)$. Let $u_k = u_{I_k}$. By Proposition~\ref{prop:Taylor} we have
\begin{align}
  |D^{\alpha} u_k(x) - D^{\alpha} u(x)|
  \le
  \sum_{a \in I} |R_2(x_a, x)| D^{\alpha} w_a^*(x)|,
\end{align}
which for ${\rm dist}(x, \partial \Omega) \ge \eps h$ can be estimated by $C \| D^2 u \|_{\infty} h_k^{2 - |\alpha|} \to 0$ as $h_k \to 0$, see Theorem~\ref{thm:: uniform error bound}.

If ${\rm dist}(x, \partial \Omega) \le \eps h$, then with the help
of Lemma~\ref{lemma:singular-boundary} precisely the same arguments
as in the proof of Theorem~\ref{thm:: uniform error bound} show that
\begin{align}
  \sum_{\substack{a \in I \\ |x - x_a| \ge Rh}} |R_2(x_a, x)| D^{\alpha} w_a^*(x)|
  \le
  C \| D^2 u \|_{\infty} h_k^{2 - |\alpha|}
  \to 0
\end{align}
for a sufficiently large constant $R > 0$. But the remaining part of
the sum vanishes for small $h_k$ as $R_2(x_a, x) = 0$ if $|x - x_a|
< Rh$, since then $u$ vanishes on a neighborhood of the segment
$\{x_a + \theta(x - x_a)\colon \theta \in [0, 1]\} \subset \{x \in
\Omega \colon {\rm dist}(x, \partial\Omega) \le (R + \eps)h_k\}$.
\end{proof}

In order to formulate our main result on density up to the boundary we  denote by $\partial_* \Omega$ the union of the interiors of the $(d-1)$-faces of $\partial \Omega$. ($\partial_* \Omega$ is the reduced boundary in the language of geometric measure theory.) The part of $\Omega$ a distance $\eps h$ away from the singular boundary $\partial \Omega \setminus \partial_* \Omega$ is denoted  $\tilde{\Omega}_{\eps h} = \{x \in \Omega \colon {\rm dist}(x, \partial \Omega \setminus \partial_* \Omega) \ge \eps h\}$.

\begin{theorem}\label{theorem:W1p-density}
Let $\Omega$ be a bounded polyhedron and $\{I_k, W_k, P_k\}$ be a sequence of LME approximation schemes satisfying the assumptions of Proposition~\ref{prop:: LME-shape functions have poly-decay} uniformly for $h_k\to 0$. Let $\eps > 0$. Then for any $u \in W^{1, p}(\Omega)$, $1 \le p < \infty$, there exists a sequence $u_k \in X_k$ such that $\| u_k - u \|_{W^{1, p}(\tilde{\Omega}_{\eps h})} \to 0$.
\end{theorem}
Therefore, density holds away from the singular boundary.

\begin{proof}
Let $u \in C^2(\Omega)$. As in the proof of Theorem~\ref{theorem:W1p0-density} we find by Proposition \ref{prop:Taylor} and the arguments in the proof of Theorem \ref{thm:: uniform error bound} that for all $x \in \tilde{\Omega}_{\eps h}$
\begin{equation}
    |D^{\alpha} u_I(x) - D^{\alpha} u(x)|
    \le
    C \| D^2 u\|_{\infty} h^{2 - |\alpha|}\left( 1 + h^{\mu} {\rm dist}^{-\mu}\left(x, \partial \Omega \right)\right),
\end{equation}
$0 < \mu < \frac{1}{2}$, where in addition we have applied Proposition~\ref{prop:blow-up-est} in order to estimate $\nabla w^*_a$ near the regular boundary. Consequently,
\begin{equation}
\begin{split}
    \int_{\tilde{\Omega}_{\eps h}} |D^{\alpha} u_I(x) - D^{\alpha} u(x)|^p \, dx
    &\le
    C h^{(2 - |\alpha|) p}
      \int_{\tilde{\Omega}_{\eps h}} 1 + h^{p \mu} {\rm dist}^{-p \mu}\left(x, \partial \Omega \right) \, dx \\
    &\le
    C h^{(2 - |\alpha|)p} \left( 1 + h^{p \mu} \int_0^1 t^{-p \mu} \, dt \right) \\
    &=
    C h^{(2 - |\alpha|)p}
\end{split}
\end{equation}
for $\mu$ sufficiently close to $0$.
\end{proof}

Note that since the $(d-2)$-dimensional Hausdorff measure of
$\partial \Omega \setminus \partial_* \Omega$ is finite, this set
has zero $2$-capacity for $d \ge 3$. Theorem
\ref{theorem:W1p-density} thus shows that $u$ can be approximated by
$u_k \in X_k$ in $H^1$ up to sets of arbitrarily small $2$-capacity.
With the help of a capacity argument we obtain from
Theorem~\ref{theorem:W1p-density}.

\begin{corollary}\label{cor:H1-density}
Let $\Omega$ be a bounded polyhedron and $\{I_k, W_k, P_k\}$ be a
sequence of LME approximation schemes, $\eps > 0$. Suppose $d > 2$.
Then for any $u \in H^{1}(\Omega)$ there exists a sequence $\chi_k
\in H^1$ with $\chi_k \to 1$ in $H^1$ and a sequence $u_k \in X_k$
such that $\chi_k u_k \to u$ in $H^1$.
\end{corollary}

\begin{proof}
Since the $(d-2)$ dimensional Hausdorff measure of the singular part $\partial \Omega \setminus \partial_* \Omega$ of the boundary is finite, this set has zero $2$-capacity:
\begin{equation}
    {\rm Cap}_2(\partial \Omega \setminus \partial_* \Omega)
    =
    0.
\end{equation}
In particular, for every neighborhood $V$ of $\partial \Omega
\setminus \partial_* \Omega$ and $\delta > 0$ there exists a
function $\psi_{\delta} \in H^1(\R^d)$ with compact support in $V$
such that $\psi_{\delta} > 1$ in a smaller neighborhood of $\partial
\Omega \setminus \partial_* \Omega$ and
\begin{equation}
    \int_{\R^d} |\nabla \psi_{\delta}|^2 \, dx
    <
    \delta.
\end{equation}
(This follows, e.~g., from Theorem~3 and its proof
in~\cite[pp.~155--157]{evansgariepy}.) By replacing, if necessary,
$\psi_{\delta}$ with a mollification of $\max \{\min\{\psi_{\delta},
1\}, -1\}$ we may assume that $\psi_{\delta}$ is smooth,
$|\psi_{\delta}| \le 1$ and in particular $\psi_{\delta} \equiv 1$
near $\partial \Omega \setminus \partial_* \Omega$.

Now suppose $u \in C^1(\bar{\Omega})$. By
Theorem~\ref{theorem:W1p-density} we find $u_k \in X_k$ with $\| u_k
- u \|_{W^{1, p}(\tilde{\Omega}_{h_k})} \to 0$ for all $p < \infty$.
Since $1 - \psi_{\delta}$ vanishes in a neighborhood of $\partial
\Omega \setminus \partial_* \Omega$, it follows that
\begin{align}
    \|(1 - \psi_{\delta}) (u_k - u) \|_{H^1(\Omega)}^2
    &\le
    \|1 - \psi_{\delta} \|_{L^{\infty}(\Omega)}^2 \| u_k - u \|_{H^1(\tilde{\Omega}_{h_k})}^2
      + \| \nabla \psi_{\delta} \|_{L^2(\Omega)} \| u_k - u \|_{L^{\infty}(\tilde{\Omega}_{h_k})}^2 \nonumber\\
    &\to 0
\end{align}
by Sobolev embedding. As
\begin{equation}
\begin{split}
  \| (1 - \psi_{\delta})u - u \|_{H^1}^2
  &=
  \|\psi_{\delta} u\|_{H^1}^2 \\
  &\le
  \int_V |u|^2 \, dx + \int_V |\nabla u|^2 \, dx + \|u\|_{\infty}^2 \int_V |\nabla \psi_{\delta}|^2 \, dx \\
  &\le
  C(|V| + \delta)
\end{split}
\end{equation}
and $V$ and $\delta$ can be chosen arbitrarily small, by choosing
diagonal sequences we see that every $u \in C^1(\bar{\Omega})$ and
hence in fact every $u \in H^1(\Omega)$ can be approximated by
sequences $(1 - \psi_{\delta_k}) u_k$, $u_k \in X_k$.
\end{proof}
\end{section}

\begin{section}{Concluding remarks}
\label{sec:: conclusions}

The preceding analysis shows that, whereas the density of the LME approximating scheme in the interior of the domain follows directly from the general results for meshfree approximation schemes, the density of the scheme up to the boundary is a matter of considerable delicacy. This situation strongly suggests relaxing the positivity constraint and allowing for signed basis functions. This relaxation is also required for the formulation of higher-order approximation schemes, as noted by \cite{arroyo_ortiz_06, cyron_09}. Indeed, in the finite-element limit shape functions of quadratic order and higher are signed functions in general. As an additional bonus, signed shape functions enable the consideration of general---not necessarily convex---domains. These extensions are pursued in a follow-up publication \cite{BompadreCyronPerottiOrtiz2011}, where LME-type approximation schemes of arbitrary order and smoothness are derived and their convergence properties are analyzed using the general analysis framework developed in this paper.

\end{section}

\bibliographystyle{plain}
\bibliography{main}
\end{document}